\documentclass[a4paper,11pt,twoside]{article}

\usepackage[utf8]{inputenc}
\usepackage[T1]{fontenc}
\usepackage[english]{babel}

\usepackage{charter}

\usepackage{indentfirst}
\usepackage{verbatim}

\usepackage{hyperref}

\usepackage[top=3.cm, bottom=4.0cm, left=2.2cm, right=2.2cm]{geometry}
\usepackage{amsmath}
\usepackage{amsthm}	
\usepackage{amsfonts}	
\usepackage{amssymb	
           ,bbm
           }

\usepackage{enumerate}
\usepackage[nobysame
           ,alphabetic
           ,initials
           ]{amsrefs}

\numberwithin{equation}{section}
 \usepackage[nodayofweek]{datetime}

\title{Differentiability of SDEs with drifts of super-linear growth}
\author{
\normalsize Peter Imkeller$^1$\\
        \small imkeller@mathematik.hu-berlin.de
\and
\normalsize Gon\c calo Dos Reis$^{2,3}$ \footnote{G. dos Reis acknowledges support from the \emph{Funda{\c c}$\tilde{\text{a}}$o para a Ci$\hat{e}$ncia e a Tecnologia} (Portuguese Foundation for Science and Technology) through the project UID/MAT/00297/2013 (Centro de Matem\'atica e Aplica\c c$\tilde{\text{o}}$es CMA/FCT/UNL).} \\ 
        \small  G.dosReis@ed.ac.uk
\and
\normalsize William Salkeld$^2$ \footnote{W. Salkeld acknowledges support from the Laura Wisewell travel fund}\\ 
        \small  w.j.salkeld@sms.ed.ac.uk 
}

\date{ 
\normalsize $^1$Humboldt Universit\"at zu Berlin\\%
    $^2$University of Edinburgh\\
		$^3$Centro de Matem\'atica e Aplica\c c$\tilde{\text{o}}$es, (CMA), FCT, UNL\\
		$^4$MIGSAA\\ [2ex]%
		\currenttime, \ddmmyyyydate\today 
		}

\theoremstyle{plain}
\newtheorem{theorem}{Theorem}[section]
\newtheorem{lemma}[theorem]{Lemma}
\newtheorem{proposition}[theorem]{Proposition}
\newtheorem{corollary}[theorem]{Corollary}

\newtheorem{definition}[theorem]{Definition}

\newtheorem{remark}[theorem]{Remark}
\newtheorem{example}[theorem]{Example}
\newtheorem{assumption}[theorem]{Assumption}

\newcommand{\bD}{\mathbb{D}}
\newcommand{\bE}{\mathbb{E}}

\newcommand{\bL}{\mathbb{L}}
\newcommand{\bN}{\mathbb{N}}
\newcommand{\bP}{\mathbb{P}}
\newcommand{\bQ}{\mathbb{Q}}
\newcommand{\bR}{\mathbb{R}}
\newcommand{\bS}{\mathbb{S}}

\newcommand{\cE}{\mathcal{E}}
\newcommand{\cF}{\mathcal{F}}

\newcommand{\cH}{\mathcal{H}}

\newcommand{\cL}{\mathcal{L}}

\newcommand{\cS}{\mathcal{S}}

\DeclareMathOperator*{\esssup}{ess\,sup}

\newcommand{\1}{\mathbbm{1}}

\begin{document}

\selectlanguage{english}

\maketitle
\renewcommand*{\thefootnote}{\arabic{footnote}}

\begin{abstract} 
We close an unexpected gap in the literature of Stochastic Differential Equations (SDEs) with drifts of super linear growth and with random coefficients, namely, we prove Malliavin and Parametric Differentiability of such SDEs. The former is shown by proving  Stochastic G\^ateaux Differentiability and Ray Absolute Continuity. This method enables one to take limits in probability rather than mean square or almost surely bypassing the potentially non-integrable error terms from the unbounded drift. This issue is strongly linked with the difficulties of the standard methodology of \cite{nualart2006malliavin}*{Lemma 1.2.3} for this setting. Several examples illustrating the range and scope of our results are presented. 

We close with parametric differentiability and recover representations linking both derivatives as well as a Bismut-Elworthy-Li formula.
\end{abstract} 

\vspace{0.3cm}



%
\footnotesize
\tableofcontents
\normalsize 

\section{Introduction}

In this manuscript we work with the class of Stochastic Differential Equations (SDEs) with drifts satisfying a super-linear growth (locally Lipschitz) and a monotonicity condition (also called one-sided Lipschitz condition); the coefficients are furthermore assumed to be random. This class of SDEs appears ubiquitously in mathematics and engineering, for example, the stochastic Ginzburg-Landau equation in the theory of superconductivity; Stochastic Verhulst equation;  Feller diffusion with logistic growth; Protein Kinetics and others, see \cite{HutzenthalerJentzenKloeden2011} and references. 

There is a wealth of results on differentiability and properties of SDEs in general. However, it is surprising that the landscape is (to the best of our knowledge) empty with respect to the superlinear growth setting apart from \cite{tahmasebi2013weak} which we discuss below. Additionally, in \cite{RIEDEL2017283} the authors discuss stochastic flows in rough path sense for a class related to ours but only up to linear growth; and using analytical tools, \cite{Cerrai2001}*{Chapter 1} and \cite{Zhang2016} require ellipticity and deterministic maps to obtain some results in the same vein as ours.  Our arguments are fully probabilistic.

\emph{Malliavin differentiability.} To establish Malliavin differentiability for an SDE with solution $X$ and with monotone drifts, the most natural path to follow is to try to apply \cite{nualart2006malliavin}*{Lemma 1.2.3} by employing a truncation procedure. This yields a sequence $X^n$ of SDEs with Lipschitz coefficients converging to $X$. Under said Lipschitz conditions the family $X^n$ is  Malliavin differentiable under suitable differentiability assumptions, with derivative $DX^n$, and one is able to appeal to \cite{nualart2006malliavin}*{Lemma 1.2.3} to conclude the Malliavin differentiability of $X$ if one is able to show that $\sup_n \bE\big[ \|D X^n\|_{H} \big]<\infty$. The truncation procedure, even smoothed out, destroys the monotonicity and, in the multi-dimensional case, it is notoriously difficult to establish the mentioned uniform bound.

To the best of our knowledge this question was studied only in \cite{tahmasebi2013weak}. The authors employ a truncation procedure in order to use \cite{nualart2006malliavin}*{Lemma 1.2.3}. Unfortunately their \cite{tahmasebi2013weak}*{Lemma 4.1} is incorrect. The constant $M_l$ presented in their equation (4.1) depends on the truncation level $n$ in a non-uniformly bounded way; the reader is invited to inspect the 2nd line of page 879. This lemma, which we were not able to fix, is used subsequently to establish the main result in \cite{tahmasebi2013weak}. 

We prove Malliavin Differentiability through a less well-known method developed by Sugita \cites{sugita1985} which uses the concepts of \emph{Ray Absolute Continuity} and \emph{Stochastic G\^ateaux Differentiability} see also the posterior developments by \cites{MastroliaPossamaiReveillac2017,imkeller2016note}. This approach is detailed in Section \ref{section:RAC+SGD} below. The merit of this method is that the limit for the Stochastic G\^ateaux derivative is a convergence in probability statement rather than a convergence in mean square statement. Put simply, this allows us to avoid cases such as the ``Witches Hat'' function where errors are non-integrable but converge to zero almost surely. 

We study the case where the coefficients of the SDE are random. We follow the ideas of \cite{geiss2016malliavin} and present two different sets of conditions which allow for Malliavin Differentiability. One set of conditions is sharp but somewhat difficult to use in practice. The other is much easier to verify but not sharp. We also provide examples discussing the scope and limitations of our approach. 

\emph{Parametric differentiability.} The second contribution of this work is parametric differentiability for SDEs of this type and in particular its implications for the classical case of deterministic coefficients. The methodology takes inspiration from the Malliavin differentiability section and we prove G\^ateaux and Fr\'echet differentiability with respect to the SDEs parameters. 

\emph{Representations, Absolute continuity of the law and Bismut-Elworthy-Li formulae.} We bridge both differentiability results by recovering (a) representation formulae linking the Malliavin derivative and the parametric one; (b) establishing  absolute continuity of the solution's Law; and (c) a Bismut-Elworthy-Li formula. 

\emph{Technical results.}
In this setting the drift term is not bounded and conditional on the coefficients' integrability the solution may not be sufficiently integrable - see Remark \ref{rem:HsarpCondsFubini} and the examples in Section \ref{sec:examples}. This means that the error terms appearing in proofs of differentiability will not be assumed to be sufficiently integrable. We negotiate this obstacle by proving everything in convergence in probability and ensuring that adequate conditions are met so that results can be lifted to the relevant setting of mean square and almost sure convergence. Proposition \ref{Proposition:GronwallConProb} contains a Gr\"onwall type inequality for the topology of Convergence in Probability that is of independent interest and is key to the methods used in this paper.
\medskip

This paper is organized as follows. In Section \ref{sec:preliminaries}, we lay out the notation and setting for this paper and recall a few baseline results from the literature. In Section \ref{sec:monotoneSDE} we prove Malliavin differentiability of SDEs of the form \eqref{eq:SDE}. There are two main results: Theorem \ref{theo-Mall-diff-monotone-SDEs} which provides a sharp method and Theorem \ref{theo:WhatsHere} which has easier to verify Assumptions but is not sharp. There is a collection of examples which explain the merits and limitations of the results we present.  In Section \ref{sec:Parametric differentiability}, we use similar methods to describe the Jacobian of the SDE. Finally, Section \ref{section:Applications} bridges Section \ref{sec:monotoneSDE} and Section \ref{sec:Parametric differentiability} and contains the so-called representations formulae and  existence and smoothness results for densities.

\section*{Acknowledgments}
The authors would like to thank C.~Geiss (U.~of Jyv\"askyl\"a), A.~Steinicke (U.~of Graz) and A.~R\'eveillac (U.~of Toulouse) for their helpful comments. In particular to the two referees whose reviews led to nontrivial improvements of the initial manuscript.

\section{Preliminaries}
\label{sec:preliminaries}

\subsection{Notation and spaces}

We denote by $\bN=\{1,2,\cdots\}$ the set of natural numbers and $\bN_0=\bN\cup \{0\}$; $\bR$ denotes the set of real numbers respectively; $\bR^+=[0,\infty)$. By $a \lesssim b$ we denote the relation $a \leq C\, b$ where $C>0$ is a generic constant independent of the relevant parameters and may take different values at each occurrence. By $\lfloor x \rfloor$ we denote the largest integer less than or equal to $x$. Let $A$ be a $d\times m$ matrix, we denote the Transpose of $A$ by $A^T$. When $A$ is a matrix, we denote $|A|$ by Tr$(A\cdot A^T)^{{1}/{2}}$. 

Let $f:\bR^d \to \bR$ be a differentiable function. Then we denote $\nabla f$ to be the gradient operator and $H[ f]$ to be the Hessian operator. $\partial_{x_i}$ is the 1st partial derivative wrt $i$-th position. $\1_A$ denotes the usual indicator function over some set $A$

We use standard big $O$ and little $o$ notation to mean that for $f_n, f>0$

\begin{align*}
 f_n = O(f) \ \  \iff \ \  \lim_{n\to \infty} \frac{f_n}{f}=C<\infty
\qquad \textrm{and}\qquad 
f_n = o(f) \ \ \iff \ \ \lim_{n\to \infty} \frac{f_n}{f}=0. 
\end{align*}
where $C$ is a constant independent of the limiting variable. 

Let $(\Omega, \cF, \bP)$ be a probability space carrying an $m$-dimensional Brownian Motion on the interval $[0,T]$; the Filtration on this space satisfies the usual assumptions. We denote by $\bE$ and $\bE[\cdot|\cF_t]$ the usual expectation and conditional expectation operator (wrt to $\bP$) respectively. For a random variable $X$ we denote its probability distribution (or Law) by $\cL^X$; the law of a process $(Y(t))_{t\in[0,T]}$  at time $t$ is denoted by $\cL_t^Y$.

Let $p\in[1,\infty)$. We introduce the following spaces and when there is no ambiguity about the underlying spaces or measures, we omit their arguments.
\begin{itemize}
\item Let $C([0,1])$ denote the space of continuous functions $f:[0,1]\to \bR$ endowed with the uniform norm $\| f \|_\infty = \sup_{s\in[0,T]} |f(s)|$ and $\|f\|_{\infty, t} = \sup_{s\in[0,t]} |f(s)|$; $C_b([0,1])$ its subspace of bounded functions; $C^k_b(\bR^m)$\index{$C^k_b(\bR^m)$} the set of $k$-times differentiable real valued maps defined on $\bR^m$ with bounded partial derivatives up to order $k$, and $C^\infty_b(\bR^m)=\cap_{k\geq 1} C_b^k(\bR^m)$; $C^0_b$ its subspace of continuous bounded functions;

\item Let $L^p([0,1])$ denote the space of functions $f:[0,1]\to \bR$ satisfying $\|f\|_p= \big(\int_0^1 |f(r)|^pdr\big)^{1/p}<\infty$. Let $H$ be the usual Cameron-Martin Hilbert space  
\begin{align*}
H= \Big\{ h(t)=\int_0^t \dot{h}(s)ds,\ t\in[0,T] ;\ h(0)=0,\ \dot{h}\in L^2([0,T])\Big\}.
\end{align*}
 
\item 

Let $L^{p}(\cF_t;\bR^d;\bQ)$\index{$L^p$}, $t\in [0,T]$, is the space of $\bR^d$-valued $\cF_t$-measurable random variables $X$ with norm  $\|X\|_{L^p} = \bE^\bQ[\, |X|^p]^{1/p} < \infty$; $L^\infty$ refers to the subset of  bounded random variables with norm $\lVert X \rVert_{L^\infty}=\esssup_{\omega\in\Omega}|X(\omega)|$; Let $L^{0}(\cF_t;\bR^d)$ be the space of $\bR^d$-valued $\cF_t$ measurable, adapted random variables with the topology of convergence in probability.

\item 
$\cS^{p}([0,T],\bR^m, \bQ)$ is the space of $\bR^d$-valued measurable $\cF$-adapted processes $(Y_t)_{t\in[0,T]}$ satisfying $\|Y \|_{\cS^p} = \bE^\bQ[\|Y\|_\infty^p]^{1/p}= \bE^\bQ[\sup_{t\in[0,T]}|Y(t)|^p]^{1/p}
<\infty$; $\cS^\infty([0,T],\bR^m, \bQ)$ refers to the intersection of $\cS^{p}([0,T],\bR^m, \bQ)$ for every $p\geq 1$. 

\item $\bD^{k,p}(\bR^d)$ and $\bL_{k,p}(\bR^d)$ the spaces of Malliavin differentiable random variables and processes, see relevant section below. Similarly, let $\bD^{k,p}(\cS^p)$ denote the space of Malliavin differentiable, $\cS^p$ valued random variables. 
\end{itemize}

\subsection{Malliavin Calculus} 

Let $\cH$ be a Hilbert space and $W:\cH \to L^2(\Omega)$ a Gaussian random variable. The space $W(\cH)$ endowed with an inner product $\langle W(h_1), W(h_2)\rangle = \bE[W(h_1) W(h_2)]$ is a Gaussian Hilbert space. Let $C_p^\infty(\bR^n; \bR)$ be the space of all infinitely differentiable function which has all partial derivatives with polynomial growth. Let $\bS$ be the collection of random variables $F:\Omega \to \bR$ such that for $n\in\bN$, $f\in C_p^\infty(\bR^n; \bR)$ and $h_i \in \cH$ can be written as $F = f(W(h_1), ..., W(h_n))$. Then we define the derivative of $F$ to be the $\cH$ valued random variable
\begin{align*}
DF = \sum_{i=1}^n \partial_{x_i} f\big( W(h_1), ..., W(h_n) \big) h_i 
= \big\langle (\nabla_x f)\big( W(h_1), ..., W(h_n) \big), h \big\rangle_{\bR^n}.
\end{align*}
In the case of a stochastic integrals, $\cH=L^2([0,T])$ and the Malliavin derivative takes the stochastic integral of a deterministic and square integrable function. 

The Malliavin derivative from $L^p(\Omega)$ into $L^p(\Omega, \cH)$ is closable and the domain of the operator is defined to be $\bD^{1, p}$. $\bD^{1, p}$ is the closure of the of the set $\bS$ with respect to the norm
\begin{align*}
\|F\|_{1, p} = \Big[ \bE[ |F|^p] + \bE[ \|DF\|_{\cH}^p] \Big]^{\tfrac{1}{p}}. 
\end{align*}

We also define the Directional Malliavin Derivative $D^hF = \langle DF, h\rangle$ for any choice of $h\in\cH$.  For more details, see \cite{nualart2006malliavin}. 

\subsubsection*{The Probability Space}

Throughout, we study the case where our filtered probability space $(\Omega, \cF, (\cF_t)_{t\in [0,T]}, \bP)$ is complete, right-continuous and contains an $m$-dimensional Brownian motion and additionally the probability subspace $(\Omega, \cF_0, \bP)$ contains a collection of random variables that are independent of the Brownian motion. Therefore, by conditioning against the $\sigma$-algebra $\cF_0$, the filtered probability space $(\Omega, \cF, (\cF_t)_{t\in [0,T]}, \bP_{\cF_0})$ can be equated with a canonical Wiener space. 

When we write $\omega\in\Omega$, this should be thought of as an element of the canonical Wiener space. In Section \ref{sec:monotoneSDE}, we perturb the SDE only on the canonical Wiener space and as the contents of $(\Omega, \cF_0, \bP)$ is orthogonal, it will be unaffected. In Section \ref{sec:Parametric differentiability}, we perturb on the space $L^p(\cF_0; \bR^d; \bP)$. 

\subsection{Existence and Uniqueness of SDE with Local Lipschitz coefficients}

We present the class of SDEs that we will be working with. 

\subsubsection*{Lipschitz and Local Lipschitz coefficients}

Let $(t,\omega,\theta )\in[0,T] \times \Omega \times L^0(\Omega, \cF_0, \bP; \bR^d)$. In this paper, we prove differentiability properties of the SDE 
\begin{equation}
\label{eq:SDE}
X_\theta(t)(\omega) = \theta + \int_0^t b\Big(s,\omega, X(s)(\omega)\Big) ds + \int_0^t \sigma\Big(s, \omega, X(s)(\omega)\Big) dW(s),
\end{equation}
driven by a $m$-dimensional Brownian motion $W$. 

\begin{assumption}
\label{Assumption:1}
Let $p\geq 2$. Let $\theta:\Omega\to\bR^d$, $b:[0,T]\times \Omega \times \bR^d \to \bR^d$ and $\sigma:[0,T]\times \Omega \times \bR^d \to \bR^{d\times m}$ be progressively measurable maps and $L>0$ such that:
\begin{itemize}
\item $\theta \in L^p(\cF_0; \bR^d; \bP)$ is independent of the Brownian motion $W$. 
\item $b$ and $\sigma$ are integrable in the sense that 
\begin{align}
\label{eq:Integrabilitybandsigma} 
\bE\Big[ \Big(\int_0^T |b(t, \omega, 0)| dt\Big)^{p}\Big],
\bE\Big[ \Big( \int_0^T | \sigma(t, \omega, 0) |^2 dt\Big)^{\tfrac{p}{2}}\Big]< \infty.
\end{align}
\item 
$\exists L$ such that for almost all $(s, \omega)\in [0,T]\times \Omega$ and $\forall x, y\in\bR^d$ we have
\begin{align*}
\big\langle x-y, b(s, \omega, x) - b(s, \omega, x)\big\rangle_{\bR^d} \leq L |x-y|^2
\quad\textrm{and}\quad
| \sigma(s, \omega, x) - \sigma(s, \omega, y) | \leq L |x-y|.
\end{align*}
\item For $x, y\in\bR^d$ such that $|x|, |y|<N$ and for almost all $(s, \omega)\in [0,T]\times \Omega$, $\exists L_N>0$ such that
\begin{align*}
| b(s, \omega, x) - b(s, \omega, y) | \leq L_N |x-y|.
\end{align*}
\end{itemize}
\end{assumption}
The next result extends other results found in the literature to the case of random coefficients. Existence and uniqueness of a solution follow the methods of \cite{mao2008stochastic}*{Theorem 2.3.6}; the case of random coefficients is not addressed there but the general methodology is applicable in the same way with only more care being taken when proving integrability. 
\begin{theorem}
\label{theorem:ExistenceUniquenessMao}
Let $p\geq 2$. Suppose Assumption \ref{Assumption:1} is satisfied. Then there exists a unique solution $(X(t))_{t\in[0,T]}$ to the SDE \eqref{eq:SDE} in $\cS^p$ and 
\begin{align*}
\bE\Big[ \| X\|_{\infty}^p\Big] \lesssim \Big( \bE\Big[ |\theta|^p\Big] 
+ \bE\Big[ \Big( \int_0^T |b(s, \omega, 0)| ds \Big)^p\Big] 
+ \bE\Big[ \Big( \int_0^T \Big| \sigma(s, \omega, 0)\Big|^2 ds \Big)^{\tfrac{p}{2}} \Big] \Big).
\end{align*}
Moreover, the map $t\mapsto X(t)(\omega)$ is $\bP$-a.s. continuous.

Finally, the solution of the SDE is Stochastically Stable in the sense that for $\forall \xi, \theta \in L^p(\cF_0; \bR^d; \bP)$, 
\begin{align*}
\bE\Big[ \| X_\xi - X_\theta \|_\infty^p \Big] \lesssim \bE\Big[ | \theta - \xi|^p\Big]. 
\end{align*}
\end{theorem}
\begin{proof}
This proof can be found in Appendix \ref{Appendix:MomentCalculations}.
\end{proof}

\begin{remark}[Issues with integrability and Fubini - Sharp conditions]
\label{rem:HsarpCondsFubini}
The integrability conditions of Assumption \ref{Assumption:1} are designed to be sharp. However, they yield processes which can have some problematic properties. 

It is very important to note that we cannot (in general) swap the order of integration at this point! This is a key point in our manuscript. We are not  able to assume that the drift term is sufficiently integrable (given \eqref{eq:Integrabilitybandsigma}) and hence the error terms appearing in the proofs of differentiability below will not be assumed to be integrable. 

To emphasize our point consider the following monotone drift function $b(t, \omega, x) = x - x^{5} $ and $\sigma(t, \omega, x)$ is chosen so that for some $t'\in[0,T)$
\begin{align*}
\bE\Big[ \big| \int_0^T|\sigma(t, \omega, 0)|^2dt\big|^2\Big]< \infty, \quad \bE\Big[ \big| \int_0^{t'}|\sigma(t, \omega, 0)|^2dt\big|^{\tfrac{5}{2} }\Big] = \infty.
\end{align*}
These satisfy the conditions of Assumption \ref{Assumption:1} for $p=4$ but not for $p=5$. We can then argue as follows: for $t\in[t',T]$ 
\begin{align*}
\bE[\, |X(t)|^4] <\infty, \quad \bE[\, |X(t)|^5] = \infty \quad \mbox{and in particular} \quad \bE\Big[ \int_{t'}^t |X(s)|^5 ds\Big] = \infty. 
\end{align*}
The existence of finite fourth moments ensures we have finite first moments and hence for $t>t'$
\begin{align*}
\bE\Big[ \int_0^t\Big( X(s) - X(s)^5 \Big)ds \Big]<\infty
\quad \textrm{which implies that}\quad
\bE\Big[ \int_{t'}^t X(s)^5 ds \Big] <\infty.
\end{align*}
\end{remark}

\subsubsection*{On SDEs with Linear Coefficients}

Let $(t,\omega,\theta )\in[0,T] \times \Omega \times L^0(\Omega, \cF_0, \bP; \bR^d)$. We will also be interested in SDEs of the form
\begin{equation}
\label{eq:SDEDeriv}
X_\theta(t)(\omega) = \theta + \int_0^t \Big[ B\big(s,\omega \big) X_\theta(s)(\omega) + b(s, \omega) \Big] ds + \int_0^t \Big[ \Sigma\big(s, \omega \big) X(s)(\omega) + \sigma(s, \omega) \Big]dW(s),
\end{equation}
driven by a $m$-dimensional Brownian motion $W$. The derivatives of SDEs of the form \eqref{eq:SDE} will satisfy linear SDEs  of the form \eqref{eq:SDEDeriv}. 

\begin{assumption}
\label{Assumption:linear}
Let $p\geq 1$. Let $B:[0,T]\times \Omega \to \bR^{d\times d}$, $\Sigma:[0,T]\times \Omega \to \bR^{(d\times m)\times d}$, $b:[0,T] \times \Omega \to \bR^d$ and $\sigma:[0,T] \times \Omega \to \bR^{d\times m}$ be progressively measurable maps such that:
\begin{itemize}
\item $\theta \in L^p(\cF_0; \bR^d; \bP)$ is independent of the Brownian motion $W$. 
\item $B$, $b$, $\Sigma$ and $\sigma$ are integrable in the sense that $\exists L\geq 0$ such that $\forall x \in \bR^d$
\begin{align*}
x^T B(t, \omega) x< L|x|^2\quad \bP\mbox{-a.s.}, \qquad \int_0^T \| \Sigma(t, \cdot) \|_{L^{\infty}}^2 dt < \infty, 
\\
\bE\Big[ \Big( \int_0^T |b(t, \omega)|dt \Big)^p \Big] ,\quad \bE\Big[ \Big( \int_0^T | \sigma(t, \omega)|^2 dt\Big)^{\tfrac{p}{2}} \Big] < \infty. 
\end{align*}
\end{itemize}
\end{assumption}

One advantage of SDEs of the form \eqref{eq:SDEDeriv} is that they have an explicit solution unlike SDEs of the form \eqref{eq:SDE} where a solution exists but cannot be explicitly stated. Linear SDEs do have Lipschitz coefficients, but their Lipschitz constants are not uniform over $(t, \omega)\in [0,T]\times \Omega$. Therefore, we cannot apply Theorem \ref{theorem:ExistenceUniquenessMao}. 

Notice that for Assumption \ref{Assumption:linear}, we do not make any requirement on $B$ being positive definite operator. In fact, we may be interested in cases where $\exists x\in\bR^d$ such that $x^T (\int_0^T B(t, \omega)dt) x=-\infty$ with positive probability. 

\begin{theorem}
\label{theorem:ExistenceUniquenessMaoLinearSDE}
Let $p\geq 1$. Suppose Assumption \ref{Assumption:linear} is satisfied. Then there exists a unique solution $(X(t))_{t\in[0,T]}$ to the SDE \eqref{eq:SDEDeriv} in $\cS^p$ with explicit form
\begin{equation*}
X_\theta(t) = \Psi(t) \Bigg( \theta + \int_0^t \Psi(s)^{-1}\Big[ b(s, \omega) - \Big\langle \Sigma(s, \omega), \sigma(s, \omega) \Big\rangle_{\bR^m}\Big] ds + \int_0^t \Psi(s)^{-1} \sigma(s, \omega) dW(s) \Bigg),
\end{equation*}
where $\Psi:[0,T]\times \Omega \to \bR^{d\times d}$ can be written as
\begin{equation}
\label{eq:LinearSDEFundamental}
\Psi(t) = I_d \exp\Bigg( \int_0^t \Big[ B(s, \omega) - \frac{\big\langle \Sigma(s, \omega), \Sigma(s, \omega)\big\rangle_{\bR^m}}{2} \Big] ds + \int_0^t \Sigma(s, \omega) dW(s) \Bigg),
\end{equation}
and
\begin{align*}
\bE\Big[ \| X_\theta\|_\infty \Big] \lesssim \Big( \bE\Big[ |\theta|^p\Big] + \bE\Big[ \Big( \int_0^T |b(s, \omega)| ds\Big)^p\Big] + \bE\Big[ \Big( \int_0^T |\sigma(s, \omega)|^2 ds \Big)^{\tfrac{p}{2}} \Big] \Big).
\end{align*}
Moreover, the map $t\mapsto X(t)(\omega)$ is $\bP$-a.s. continuous.

Finally, the solution of the equation is Stochastically stable in the sense that $\forall \xi, \theta \in L^p(\cF_0; \bR^d; \bP)$
\begin{align*}
\bE\Big[ \| X_\xi - X_\theta\|_{\infty}^p\Big] \lesssim \bE\Big[ | \xi - \theta|^p \Big].
\end{align*}
\end{theorem}

\begin{proof}
An existence and uniqueness proof is found in \cite{mao2008stochastic}*{Theorem 3.3.1}. Moment calculations are proved in Appendix 
\ref{Appendix:MomentCalculations}. Stochastic stability is proved in the same fashion as in Theorem \ref{theorem:ExistenceUniquenessMao}. 
\end{proof}

\subsection{A Gr\"onwall inequality}

To the best of our knowledge the next result is new and of independent interest. While unsurprising, this is key to the methods of this paper.

\begin{proposition}[Gr\"onwall Inequality for the Topology of Convergence in Probability]
\label{Proposition:GronwallConProb}
Let $n\in\bN$, $A_n:[0,T] \times \Omega \to \bR$ be a sequence of adapted stochastic processes such that $\|A_n\|_{\infty} \xrightarrow{\bP} 0$ as $n\to \infty$. 

Let $U_n$ be the solution of the SDE 
\begin{align*}
U_n(t) = A_n(t) + \int_0^t f(U_n(s)) ds + \int_0^t g(U_n(s)) dW(s), \qquad t\in[0,T]
\end{align*}
where $f,g:\bR \to \bR$ are Monotone growth and Lipschitz respectively (see 3rd bullet point of Assumption \ref{Assumption:1}) and $f(0)=g(0)=0$. 

Then $\|U_n\|_{\infty} \xrightarrow{\bP} 0$ as $n\to \infty$. 
\end{proposition}

Notice that since we do not have finite second moments of $\|A_n\|_{\infty}$, we cannot prove this using a mean square type argument. 

\begin{proof}
Fix $\delta>0$ and let $n\in\bN$. We have that
\begin{align*}
\bP\Big[ \|U_n\|_\infty> \delta\Big] &\leq \bP\Big[ \|U_n\|_\infty> \delta, \|A_n\|_\infty \leq \eta\Big] + \bP\Big[ \|A_n\|_\infty>\eta\Big],
\end{align*}
for any choice of $\eta>0$. We already have that $\lim_{n\to \infty} \bP\big[ \|A_n\|_\infty >\eta\big] = 0$ for any choice of $\eta>0$ by assumption. Define the sequence of stopping times $ \tau_n=\inf \{ t'>0: |A_n(t')| > \eta\}$, $n\in\bN$.

Firstly, we show that $\lim_{n\to\infty} \tau_n\geq T$ almost surely. Suppose this was not the case. Then $\exists \Omega' \subset \Omega$ with $\bP(\Omega')>0$ and $\forall \omega\in\Omega'$ $\exists n_k(\omega)$ an increasing subsequence of integers such that $\tau_{n_k}(\omega)<T$ for all $k\in \bN$.  Then $\forall \omega\in\Omega'$, $\|A_{n_k}\|_{\infty}(\omega)>\eta$ for all $k\in\bN$. But that implies that for any $k\in \bN$ we have
\begin{align*}
\Omega' \subset \{ \omega \in\Omega; \|A_{n_k}\|_\infty(\omega) >\eta\}
\quad\textrm{and hence that}\quad
\bP\big[ \|A_{n_k}\|_\infty>\eta\big] > \bP[\Omega'].
\end{align*}
The latter contradicts the assumption that $\|A_{n_k}\|_{\infty}$ converges to $0$ in probability. So any such set $\Omega'$ must have measure 0 and we conclude $\lim_{n\to \infty} \tau_n>T$ almost surely. 

The SDE for $U_n(t)$ is well defined for $t\in[0,\tau_n]$. Outside of this interval, $A_n$ may not be integrable so we may not be able construct a solution. However $\forall \omega\in \Omega$ such that $\|A_n\|_\infty (\omega) \leq \eta$ we have that $\tau_n(\omega)>T$. Therefore
\begin{align*}
\bP\Big[ \|U_n(\cdot)\|_\infty >\delta, \|A_n\|_\infty \leq \eta\Big] = \bP\Big[ \|U_n(\cdot\wedge \tau_n)\|_\infty >\delta, \|A_n\|_\infty \leq \eta\Big]
\end{align*}
because the process $U_n(\cdot)$ and the stopped process $U_n(\cdot\wedge \tau_n)$ are $\bP$-almost surely equal when one restricts to the event where $\| A_n\|_\infty\leq \eta$. 

As we know that the solution $U_n(t\wedge \tau_n)$ will exist and make sense, it serves to introduce this stopping time. Thus we get
\begin{align*}
\bP\Big[ \|U_n\|_\infty> \delta\Big] &\leq \bP\Big[ \|U_n(\cdot\wedge \tau_n)\|_\infty >\delta, \|A_n\|_\infty \leq \eta\Big]
+ \bP\Big[ \|A_n\|_\infty>\eta\Big]
\\
&\leq\bP\Big[ \|U_n(\cdot\wedge \tau_n)\|_\infty >\delta\Big]
+ \bP\Big[ \|A_n\|_\infty>\eta\Big].
\end{align*}
Now we consider the SDE for $U_n(t\wedge \tau_n)$. The stopping time prevents the term $A_n(t\wedge \tau_n)$ from getting any larger that $\eta$ and ensures that the stochastic integral is a local martingale. 

Hence we apply Theorem \ref{theorem:ExistenceUniquenessMao} to obtain existence/uniqueness of the solution and moment bounds.

\begin{align*}
\bE\Big[ \| U_n(\cdot \wedge \tau_n)\|_{\infty}^2\Big] < \eta^2 e^C
\quad \textrm{and therefore}\quad
\bP\Big[ \|U_n\|_\infty> \delta\Big] \leq \frac{\eta^2 e^C}{\delta^2} + \bP\Big[ \|A_n\|_\infty>\eta\Big].
\end{align*}
Choose $\eta$ such that ${\eta^2 e^C}/{\delta^2}<{\varepsilon'}/{2}$. Then find $N\in \bN$ such that $\forall n\geq N$ $\bP\big[ \|A_n\|_\infty>\eta\big]< {\varepsilon'}/{2}$. This concludes the proof. 
\end{proof}


\section{Malliavin Differentiability of SDEs with monotone coefficients}
\label{sec:monotoneSDE}

In this section we prove two Malliavin differentiability result for SDEs in the class given by Assumption \ref{Assumption:1}. We use a less known method using the concepts of \emph{Ray absolute continuity} and \emph{Stochastic G\^ateaux Differentiability}  initiated by \cites{sugita1985} and later developed by \cites{MastroliaPossamaiReveillac2017,imkeller2016note}.

For SDEs of the form \eqref{eq:SDE}, the proof of existence and uniqueness of a solution involves a sequence of random variables which converge almost surely to the solution rather than in mean square. Indeed this sequence of random variables does not converge in mean square, unlike in the proof of Existence and Uniqueness for SDEs with Lipschitz coefficients. This means that the classical method from \cite{nualart2006malliavin}*{Lemma 1.2.3} cannot be applied; recall further our observation on the role that Proposition \ref{Proposition:GronwallConProb} will play here. 

\subsection{Main results and their assumptions}
We state the main assumptions and results with the proofs postponed for later sections.
\begin{assumption}
\label{Assumption:2}
Let $b: [0,T]\times \Omega\times \bR^d \to \bR^d$ and $\sigma: [0,T] \times \Omega\times\bR^d \to \bR^{d\times m}$ satisfy Assumption \ref{Assumption:1} for some $p>2$. Further, suppose
\begin{enumerate}[(i)]
\item For almost all $(t, \omega)\in [0,T]\times \Omega$ the functions $\sigma(t, \omega, \cdot)$ and $b(t, \omega, \cdot)$ have spatial partial derivatives in all directions. 
\item For all $h\in H$ and $(\varepsilon, x)\in \bR^+ \times \bR^d$, we have that the maps $\bR^+ \times \bR^d \to L^0(\Omega)$ 
\begin{align*}
(\varepsilon, x) \mapsto \int_0^T \Big| \nabla_x \sigma(t, \omega+\varepsilon h, x) \Big|^2 dt
\qquad\textrm{and}\qquad 
(\varepsilon, x) \mapsto \int_0^T \Big| \nabla_x b(t, \omega+\varepsilon h, x) \Big|^2 dt,
\end{align*}
are jointly continuous (where convergence in $L^0$ means convergence in probability).

\item $\exists U:[0,T]^2\times \Omega \to \bR^{d\times m}$ and $V:[0,T]^2\times \Omega \to \bR^{(d\times m)\times m}$ which satisfy that for $s>r$ $U(s, r, \omega) = V(s, r, \omega) = 0$ and 
\begin{align*}
\bE\Big[ \Big( \int_0^T \Big( \int_0^T \Big|U(s, r, \omega)\Big|^2ds\Big)^{\tfrac{1}{2}}  dr\Big)^p \Big] <\infty
\quad \textrm{and}\quad
\bE\Big[ \Big( \int_0^T \int_0^T \Big|V(s, r, \omega) \Big|^2ds dr\Big)^{\tfrac{p}{2}} \Big]<\infty.
\end{align*}
\item $b$ and $\sigma$ satisfy, as $\varepsilon \to 0$,  that $\forall h \in H$  
\begin{align*}
\bE\Big[ \Big( \int_0^T \Big| \frac{b(r, \omega + \varepsilon h, X(r)) - b(r, \omega, X(r))}{\varepsilon} - \int_0^r U(s, r, \omega)\dot{h}(s)ds\Big| dr \Big)^2 \Big] \to 0,
\\
\bE\Big[  \int_0^T \Big| \frac{\sigma(r, \omega+ \varepsilon h, X(r)) - \sigma(r, \omega, X(r))}{\varepsilon} - \int_0^r V(s, r, \omega) \dot{h}(s)ds\Big|^2 dr \Big] \to 0.
\end{align*}
\end{enumerate}
\end{assumption}
In the above condition neither $b$ or $\sigma$ are assumed to be in $\bD^{1,2}$, they are only assumed to be Malliavin differentiable over the sub-manifold on which $X$ (solution to \eqref{eq:SDE}) takes values on. After our main results we give examples of SDE illustrating the scope of our assumptions. 

\begin{theorem}[Malliavin Differentiability of Monotone SDEs]
\label{theo-Mall-diff-monotone-SDEs}
Take $p>2$. Let Assumption \ref{Assumption:2} hold and denote by $X$ the unique solution of the SDE \eqref{eq:SDE} in $\cS^p$. 

Then $X$ is Malliavin differentiable, i.e.~$X\in \bD^{1,p}(\cS^p)$ and there exist adapted processes $U$ and $V$ such that the Malliavin derivative satisfies for $0\leq s\leq t \leq T$
\begin{align}
\label{eq:SDEMallDeriv}
D_s X(t)(\omega) =
& 
\sigma(s,\omega, X(s)(\omega)) +  \int_s^t U(s, r, \omega) dr + \int_s^t V(s, r, \omega) dW(r)
\\
\nonumber
&+\int_s^t \nabla_x b(r,\omega, X(r)(\omega)) D_s X(r)(\omega) dr + \int_s^t \nabla_x \sigma(r,\omega, X(r)(\omega)) D_s X(r)(\omega) dW(r),
\end{align}
and otherwise $D_s X(t)=0$ for $s>t$.
\end{theorem}
The proof of Theorem \ref{theo-Mall-diff-monotone-SDEs} can be found in Section \ref{sec:Proof of 1st MallDiff main theorem}.
\begin{remark}[Notation]
At the simplest level, we have $X$ is $\bR^d$-valued and $W$ is $\bR^m$-valued. Therefore $b,\sigma$ are $\bR^d$- and $\bR^{d\times m}$-valued respectively. Hence we have the collection of one-dimensional SDEs
\begin{align*}
X^{(i)}(t)(\omega) = \theta^{(i)} + \int_0^t b^{(i)} (s, \omega, X(s)(\omega)) ds + \sum_{j=1}^m \int_0^t \sigma^{(i, j)}(s, \omega, X(s)(\omega)) dW^{(j)} (s), 
\end{align*}
where $i$ is an integer between 1 and $d$. 

The Malliavin Derivative $DX$ is therefore a $\bR^{d\times m}$ valued process and we get the system of equations
\begin{align*}
D_s^{(k)} X^{(i)} (t)(\omega) =& \sigma^{(i, k)}(s, \omega, X(s)(\omega))ds 
\\
& + \int_s^t U^{(i, k)}(s, r, \omega) dr + \sum_{j=1}^m \int_s^t V^{(i, j, k)} (s, r, \omega) dW^{(j)} (r) 
\\
&+ \int_s^t \Big\langle (\nabla_x b^{(i)})(r, \omega, X(r)(\omega)), D_s^{(k)}X(t)(\omega) \Big\rangle_{\bR^d} dr
\\
&+ \sum_{j=1}^m \int_s^t \Big\langle (\nabla_x\sigma^{(i, j)})(r, \omega, X(r)(\omega)), D_s^{(k)}X(t)(\omega) \Big\rangle_{\bR^d} dW^{(j)}(r), 
\end{align*}
for $i$ an integer between $1$ and $d$ and $k$ an integer between $1$ and $m$. 
\end{remark}

\begin{remark}[Mollification and non-differentiability of $b$ and $\sigma$]
\label{rem:mollification}
Using classic mollification arguments the assumptions of Theorem \ref{theo-Mall-diff-monotone-SDEs} concerning the behaviour of $x\mapsto b(\cdot,\cdot,x)$ and $x\mapsto \sigma(\cdot,\cdot,x)$ can be further weakened. Namely, $\sigma$ can be assumed to be uniformly Lipschitz as opposed to continuously differentiable and $b$ can be assumed to have left- and right-derivatives not necessarily equal to each other at every point.

Under these conditions, a canonical mollification argument allows to re-obtain Theorem \ref{theo-Mall-diff-monotone-SDEs} where in \eqref{eq:SDEMallDeriv} one replaces $\nabla_x b$ and $\nabla_x \sigma$ by two processes corresponding to their generalized derivatives. 
\end{remark}

If $b$ and $\sigma$ are assumed deterministic then one immediately obtains the familiar result.
\begin{corollary}[Deterministic coefficients case]
Suppose that $b:[0,T]\times \bR^d \to \bR^d$ and $\sigma:[0,T] \times \bR^d \to \bR^{d\times m}$ satisfy Assumption \ref{Assumption:1}. Further, suppose that $x\mapsto b(\cdot, x)$ and $x\mapsto \sigma(\cdot,x)$ are continuously differentiable in their spatial variables (uniformly in $t$). 

Then $X$ is Malliavin differentiable and $D_s X(t)=0$ for $T\geq s>t\geq 0$ while for $0\leq s\leq t\leq T$
\begin{align*}
D_s X(t)(\omega) = \sigma(s, X(s)(\omega)) 
&+\int_s^t \nabla_x b(r, X(r)(\omega)) D_s X(r)(\omega) dr
\\
\nonumber
&\qquad  \qquad 
+ \int_s^t \nabla_x \sigma(r, X(r)(\omega)) D_s X(r)(\omega) dW(r).
\end{align*}
\end{corollary}
Assumption \ref{Assumption:2} is sharp for our construction, nonetheless, it can be slightly strengthened to Assumption \ref{Assumption:3} which is much easier to verify. 

\begin{assumption}
\label{Assumption:3}
Let $b: [0,T]\times \Omega\times \bR^d \to \bR^d$ and $\sigma: [0,T] \times \Omega\times\bR^d \to \bR^{d\times m}$ satisfy Assumption \ref{Assumption:1} for $p>2$. Further, suppose Assumption \ref{Assumption:2} (i) and (ii) hold and
\begin{enumerate}[(i')]
\setcounter{enumi}{2}
\item $b$ and $\sigma$ are Malliavin differentiable in the sense that $\forall x\in \bR^d$, $b(\cdot, \cdot, x) \in \bD^{1,p}(L^1([0,T]; \bR^d))$ and $\sigma(\cdot, \cdot, x) \in \bD^{1,p}(L^2([0,T]; \bR^{d\times m}))$,
\item The Malliavin derivatives of $b$ and $\sigma$ are progressively measurable and Lipschitz in their spacial variables i.e.~$\exists L>0$ constant such that $\forall (s, t, \omega) \in  [0,T]^2 \times \Omega$ and $x, y\in \bR^d$, $\bP$-almost surely
\begin{align*}
|D_s b(t, \omega, x) - D_sb(t, \omega, y)| &\leq L |x-y|, 
\\
|D_s \sigma(t, \omega, x) - D_s\sigma(t, \omega, y)| &\leq L |x-y|. 
\end{align*}
\end{enumerate}
\end{assumption}
The second main result of the section is the following theorem.
\begin{theorem}
\label{theo:WhatsHere}
Let $p>2$. Let Assumption \ref{Assumption:1} hold and denote by $X$ the unique solution of the SDE \eqref{eq:SDE} in $\cS^p$. Let $b$ and $\sigma$ satisfy Assumption \ref{Assumption:3}. Then the conclusion of Theorem \ref{theo-Mall-diff-monotone-SDEs} still holds: $X\in \bD^{1,p}(\cS^p)$ and $DX$ satisfies $D_s X(t)=0$ for $T\geq s>t\geq 0$ while for $0\leq s\leq t \leq T$
\begin{align}
\label{eq:SDEMallDeriv2}
D_s X(t)(\omega) =
& 
\sigma(s,\omega, X(s)(\omega)) 
+  \int_s^t (D_s b)(r, \omega, X(r)(\omega)) dr 
+ \int_s^t (D_s \sigma)(r, \omega, X(r)(\omega)) dW(r)
\\
\nonumber
&+\int_s^t \nabla_x b(r,\omega, X(r)(\omega)) D_s X(r)(\omega) dr + \int_s^t \nabla_x \sigma(r,\omega, X(r)(\omega)) D_s X(r)(\omega) dW(r).
\end{align}
\end{theorem}
The proof can be found in Section \ref{sec:An Example of such a differentiable SDE}. We point out that the mollification Remark \ref{rem:mollification} applies to this result as well.

It is a well documented fact, see \cite{nualart2006malliavin}*{Theorem 2.2.1}, that if one has a SDE with deterministic and Lipschitz drift and diffusion coefficients then the Malliavin derivative is the solution of a homogeneous linear SDE. Both the SDE and the Malliavin Derivative have finite moments of all orders. Therefore the solution of the SDE exists in $\bD^{1, \infty}$. 

We study the case where the coefficients are random. SDEs of this kind do not always have finite moments of all orders, and the same will apply for the Malliavin derivative. In fact, the integrability of the derivative comes directly from the integrability of the Malliavin derivatives of $b$ and $\sigma$.


\subsection{Overview of the methodology and results on Wiener spaces}
\label{section:RAC+SGD}

It is important to note that the solution of an SDE is not continuous with respect to $\omega \in \Omega$. As the SDE exists in a probability space with the filtration generated by an $m$-dimensional Brownian motion, $\omega$ can be interpreted to mean the path of an individual Brownian motion plus any extra information about what happens when $t=0$. However, it will be shown that the random variables are continuous, and indeed differentiable, when perturbed with respect to a path out of the Cameron Martin space. Hence for this section we take $h\in H^{\otimes m}$, an $m$-dimensional Cameron Martin path and $\dot{h}$ to be its derivative unless stated otherwise. We will not emphasize the difference between $H$ and $H^{\otimes m}$ in this paper.

We start by introducing the concepts of \emph{Ray absolute continuity} and \emph{Stochastic G\^ateaux Differentiability} and the results yielding Malliavin differentiability under those properties.

Let $E$ be a separable Banach space. Let $L(H, E)$ be the space of all bounded linear operators $V:H\to E$.
\begin{definition}[Ray Absolutely Continuous map]
\label{definition:RayAbsoluteContinuity}
A measurable map $f:\Omega \to E$ is said to be \emph{Ray Absolutely Continuous} if $\forall h\in H$, $\exists$ a measurable mapping $\tilde{f}_h: \Omega \to E$ such that
\begin{align*}
\tilde{f}_h(\omega) = f(\omega) \quad \bP\mbox{-a.e.}
\end{align*}
and that $\forall \omega \in \Omega$, 
\begin{align*}
t \mapsto \tilde{f}_h(\omega + th) \quad \mbox{is absolutely continuous on any compact subset of $\bR$. }
\end{align*}
\end{definition}

\begin{definition}[Stochastically G\^ateaux differentiable]
\label{definition:StochasticallyGateauxDiff}
A measurable mapping $f:\Omega \to E$ is said to be \emph{Stochastically G\^ateaux differentiable} if there exists a measurable mapping $F:\Omega \to L(H, E)$ such that $\forall h\in H$, 
\begin{align*}
\frac{f(\omega+\varepsilon h) - f(\omega)}{\varepsilon} \xrightarrow{\bP} F(\omega)[h] \quad \mbox{as $\varepsilon \to 0$.}
\end{align*}
\end{definition}
Malliavin differentiability follows from \cite{sugita1985}*{Theorem 3.1} which was later improved upon by \cite{MastroliaPossamaiReveillac2017}*{Theorem 4.1}. We recall both results next.

\begin{theorem}[\cite{sugita1985}] 
\label{theorem:SugitaResult}
Let $p>1$. The space $\bD^{1, p}(E)$ is equivalent to the space of all random variables $f:\Omega \to E$ such that $f\in L^p(\Omega; E)$ is Ray Absolutely Continuous, Stochastically G\^ateaux differentiable and the Stochastic G\^ateaux derivative $F:\Omega \to L(H, E)$ is $F\in L^p(\Omega; L(H, E))$. 
\end{theorem}

\begin{remark}
We know from standard references such as \cite{UestuenelZakai2000} that the map $t \mapsto \tilde{f}_h(\omega + th)$ is continuous as a map from $[0,1]\to L^0(\Omega)$. The point of proving the stronger absolute continuity is to find a representation of the form
$$
\tilde{f}_h(\omega + \varepsilon h) -  \tilde{f}_h(\omega) = \int_0^\varepsilon F(\omega+rh)[h] dr,
$$
where the object $F(\omega)$ is a candidate for the Malliavin Derivative. Proving Stochastic G\^ateaux Differentiability is then verifying that this object is a bounded linear operator and allows one to extend from G\^ateaux to Fr\'echet. Thus a random variable which is Ray Absolutely Continuous but not Stochastic G\^ateaux Differentiable has a Malliavin Directional Derivative in all directions, but there is a sequence of elements $h_n \in H$ such that $F(\omega)[h_n] \to \infty$. 

By contrast, if one has Stochastic G\^ateaux Differentiability but not Ray Absolute Continuity, then one can prove existence of the Malliavin Derivative but which is not in $L^1(\Omega)$ e.g. $\bE[ \| F(\omega)\|_{L(H, E)} ] = \infty$. 
\end{remark}

\begin{definition}[Strong Stochastically G\^ateaux differentiable]
\label{def:Reveillacs results}
Let $p>1$. A random variable $f\in L^p(\Omega; E)$ is said to be \emph{Strong Stochastically G\^ateaux differentiable} if there exists a measurable mapping $F:\Omega \to L(H, E)$ such that $\forall h\in H$
\begin{equation}
\label{eq:StrongStochasticGatDif}
\lim_{\varepsilon \to 0} \bE\Big[ \Big\| \frac{f(\omega+\varepsilon h) - f(\omega)}{\varepsilon} - F(\omega)[h]\Big\| \Big] \to 0.
\end{equation}

\end{definition}

\begin{theorem}[\cite{MastroliaPossamaiReveillac2017}]
\label{theo:Reveillacs results}
Let $p>1$. The space $\bD^{1, p}(E)$ is equivalent to the space of all random variables $f\in L^p(\Omega; E)$ which are Strong Stochastically G\^ateaux differentiable. 
\end{theorem}

The merit of \cite{sugita1985} is that it allows one to prove Malliavin differentiability by first establishing existence of a G\^ateaux derivative and then extending to the full Frech\'et derivative. The convergence of the G\^ateaux derivative in probability is a very weak condition that is much easier to prove than full Malliavin differentiability.

\cite{MastroliaPossamaiReveillac2017} extends this result to the stronger Strong Stochastic G\^ateaux Differentiability condition and removed the Ray Absolute Continuity condition. 

Both of these methods have their merits. While studying different examples of processes with monotone growth, we became interested in the particular example where the drift term has polynomial growth of order $q$ but only finite moments up to $p<q-2$. In this case, one cannot in general find a dominating function for the error terms coming from the drift of the SDE while trying to prove Stochastic G\^ateaux Differentiability. It therefore became necessary to prove only a convergence in probability statement. 

\begin{corollary}
\label{corollary:HowToRAC}
Suppose a measurable map $f:\Omega \to E$ is Stochastically G\^ateaux Differentiable and additionally that for $\delta>0$
\begin{equation}
\label{eq:SSGD=RAC}
\sup_{\varepsilon\leq 1} \bE\Bigg[  \Big| \frac{ f(\omega + \varepsilon h) - f(\omega) }{\varepsilon}\Big|^{1+\delta} \Bigg]<\infty. 
\end{equation}
Then $f$ is Malliavin Differentiable (and so $f$ is Ray Absolutely Continuous). 
\end{corollary}

\begin{proof}
Equation \eqref{eq:SSGD=RAC} implies that the collection of random variables $\big( ({f(\omega+\varepsilon h) - f(\omega))}/{\varepsilon}\big)_{\varepsilon\leq1}$ is uniformly integrable. Stochastic G\^ateaux Differentiability means that this collection of random variables converges in probability to a limit. Since $\delta>0$, we conclude that the sequence of random variables converges in mean, or equivalently we have Strong Stochastic G\^ateaux differentiability. Theorem \ref{theo:Reveillacs results} shows this is equivalent to Malliavin Differentiability and Theorem \ref{theorem:SugitaResult} implies we must have Ray Absolute Continuity. 
\end{proof}

The convergence conditions on $U$ and $V$ in Assumption \ref{Assumption:2}(iii) and (iv) could equivalently been stated in terms of a \emph{Ray Absolute Continuity} and \emph{Stochastic G\^ateaux Differentiability} criterion instead of \emph{Strong Stochastic G\^ateaux Differentiability}. 

\subsubsection*{Classical results on the Cameron Martin transforms}

We recall two useful results from \cite{UestuenelZakai2000}, but first we introduce the notation for a Dol\'eans-Dade exponential over $[0,T]$ of some sufficient integrable $\bR^m$-valued process, $(M(t))_{t\in[0,T]}$, namely, we define for $t\in[0,T]$ and an $m$-dimensional Brownian motion $W$,
\begin{align}
 \label{eq:DDexponential}
	\cE(M)(t)& = \exp\Big( \int_0^t M(s) dW(s) - \cfrac12\int_0^t {|M(s)|^2} ds \Big).
\end{align}

\begin{proposition}[The Cameron-Martin Formula -- \cite{UestuenelZakai2000}*{Appendix B.1}]
\label{proposition:CamMarForm1}
Let $F$ be an $\cF_T$-measurable random variable. For $h\in H$ let $\cE(\dot h)(\cdot)$ be the associated Dol\'eans-Dade exponential. 

Then, when both sides are well defined, 
\begin{align*}
\bE\big[\, F(\omega + h)\,\big] 
=  \bE\Big[ F \exp\Big( \int_0^T \dot{h}(s) dW(s) - \frac12\int_0^T {|\dot{h}(s)|^2} ds\Big) \Big] 
= \bE\big[\, F(\omega)\cE(\dot h)(T)\,\big] . 
\end{align*}
Moreover, $\forall h\in H$ and $\forall p\geq 1$ that
$
\cE(\dot h)(\cdot) \in \cS^p([0,T]). 
$
\end{proposition}


\begin{proposition}[Continuity of the Cameron Martin Transform -- \cite{UestuenelZakai2000}*{Lemma B.2.1}]
\label{proposition:ContinuityCamMartin}
The map $\tau_h:[0,1] \to L^0(\Omega)$ defined by $t\mapsto f(\omega + t h)$ is continuous map from a compact interval of the real line to a measurable function with respect to the topology of convergence in probability. 
\end{proposition}


\subsection{Examples}
\label{sec:examples}

In this section, we discuss some interesting examples which emphasize the scope and sharpness of the assumptions made.
 
\begin{example}[Concerning the continuity of $s\mapsto D_s X(\cdot)$]
Previous works on Malliavin calculus, see for example \cite{nualart2006malliavin}, treat the solution of this SDE as being continuous in $s$. While this is true for those examples studied, it is not true in the general case that we study here. We only have that it is square integrability; this example shows that it is not necessary for the derivative to be continuous in $s$. Take $g\in L^2([0,T])$ be a deterministic discontinuous function (a step function would be adequate) and assume the one dimensional setting. Consider $\sigma$ of the form
\begin{align*}
\sigma(t, \omega, x) = x + \int_0^t g(s) dW(s) 
\quad\textrm{and}\quad  
b(t, \omega, x) = 0.
\end{align*}
Hence $X(t)$ satisfies 
$
X(t) = 1 + \int_0^t \big[ X(s) + \int_0^s g(r)dW(r)\big] dW(s).
$
It can be shown that the explicit solution of this equation is
\begin{align*}
X(t) = \exp\Big(W(t) - \frac{t}{2}\Big) \Big[ 1 
& - \int_0^t\int_0^r \exp\big(\frac{r}{2} - W(r)\big) g(u)du dr 
\\
&+ \int_0^t \int_0^r \exp\big(\frac{r}{2} - W(r)\big) g(u)dW(u) dr \Big].
\end{align*}
Note that, as expected, $X$ is a continuous process. 

The process $V$, which represents the Malliavin derivative of $\sigma$, is 
\begin{align*}
V(s, t, \omega) = D_s \sigma(t, \omega, X(t)(\omega)) = g(s) \1_{(0,t)}(s) 
\quad \Rightarrow \quad
\int_s^t V(s, r, \omega) dW(r) = g(s)[ W(t) - W(s)].
\end{align*}
Clearly, the latter map is not continuous in $s$. The Malliavin derivative of $X$ solves
\begin{align*}
D_sX(t) = X(s) + \int_0^s g(r) dW(r) + g(s)[ W(t) - W(s)] + \int_s^t D_sX(r) dW(r).
\end{align*}
Define $J_s(t) = \exp\Big( [W(t) - W(s)] - \tfrac{t-s}{2}\Big)$. Then the Malliavin derivative has the explicit solution
\begin{align*}
D_sX(t) = J_s(t)\Big[ X(s) + \int_0^s g(r) dW(r) + g(s) \Big( \int_s^t J_s(r)^{-1} dW(r) - \int_s^t J_s(r)^{-1} dr\Big) \Big].
\end{align*}
Since $g$ is assumed not to be continuous, this will also not be continuous in $s$. 
\end{example}

We present a case where the coefficients are not Malliavin differentiable in general but are only differentiable on the set where 
the solution $X$ takes its values. In other words,  Assumption \ref{Assumption:2} is satisfied but Assumption \ref{Assumption:3} is not.

\begin{example}[Malliavin Differentiable on the right manifold]

Let $d=m=1$ for simplicity. Let $b(t, \omega, x) = -x$ and
\begin{equation*}
\sigma(t, \omega(t), x) =\begin{cases}
(x-1)^2(x+1)^2&,  x\in[-1,1]\\
\phi(x)\cdot f(\omega(t))&,  |x|>1
\end{cases},
\end{equation*}
where $\phi \in C^\infty$, $\phi(x)=0$ for $|x|\leq1$ and $\phi(x)=1$ for $|x|\geq 2$. The function $f$ is any function $f:\bR\to\bR$ which is bounded, continuous but not differentiable and $\omega$ is the path of the Brownian motion. 

An example of such a function $f$ could be
\begin{equation*}
f(x) =\begin{cases}
W'(x)&, x\in[-1,1]\\
-2&, |x|>1
\end{cases},
\end{equation*}
where $W'(x)$ is the Weierstrass function. The Weierstrass function is continuous but not differentiable anywhere and satisfies $W'(-1)=W'(1)=-2$. The latter implies that $f$ is continuous. Hence $f(\omega(t))$ will not be Malliavin differentiable but $\varepsilon \mapsto f(\omega(t)+\varepsilon h(t))$ will be continuous. 

The derivative of $\sigma$ will satisfy
\begin{equation*}
{\partial_x \sigma}(t,\omega, x) = \begin{cases}
4x(x-1)(x+1)&, x\in[-1,1] \\
\phi'(x)\cdot f(\omega(t))&, 1<x<2 \\
0&,  |x|>2
\end{cases},
\end{equation*}
so since $f$ is bounded, we conclude that $\sigma$ is Lipschitz $\forall \omega \in \Omega$ and differentiable. 

When the initial conditions determine that the process starts inside the interval $[-1, 1]$, this is a so-called Wright-Fisher process (see \cite{mode2012stochastic}) and the solution will remain within the interval $[-1, 1]$ with probability 1. This is important because the non-Malliavin Differentiability only affects the system when the process exits the $[-1,1]$ interval. The conditions of Assumption \ref{Assumption:2} are satisfied but $\sigma(\cdot , x)$ is not Malliavin differentiable for all $x\in \bR^d$. 
\end{example}

\begin{remark}[The square-integrability case]
\label{remark:SquareIntegrabilityProblems}
In \cite{MastroliaPossamaiReveillac2017}, it is proved that one does not require the Ray Absolute Continuity condition if one can prove a \emph{Strong Stochastic G\^ateaux Differentiability} condition, see Theorem \ref{theo:Reveillacs results} and Equation \eqref{eq:StrongStochasticGatDif}. However, in \cite{imkeller2016note}, the authors provide a random variable $Z\in \bD^{1,2}$ which is not Strong Stochastic G\^ateaux differentiable in the sense that
\begin{align*}
\bE\Big[ \Big| \frac{Z(\omega + \varepsilon h) - Z(\omega)}{\varepsilon} - D^hZ\Big|^2\Big] \nrightarrow 0,\quad \textrm{as }\ \varepsilon \to 0. 
\end{align*}
It is however true that for all values $q\in[1, 2)$
\begin{align*}
\bE\Big[ \Big| \frac{Z(\omega + \varepsilon h) - Z(\omega)}{\varepsilon} - D^hZ\Big|^q\Big] \rightarrow 0,\quad \textrm{as }\ \varepsilon \to 0.
\end{align*}
In our framework, it is necessary to study the square of incremements of the process due to the nature of the monotonicity property. Therefore we require that our SDE has finite moment of order $p$ for some $p>2$. However, in light of the example provided in \cite{imkeller2016note}, we believe (but do not show) that that there exists a case where the solution to an SDE of the form \eqref{eq:SDE} which has finite moments of order up to $p= 2$ which is Malliavin Differentiable. Stochastic G\^ateaux Differentiability would follow as before, but it was unclear to us how one would prove Ray Absolute Continuity of such a process. 
\end{remark}

\begin{remark}[The spatial Lipschitz condition for the Malliavin Derivatives of $b$ and $\sigma$]
\label{rem:cautionarytale}
In Assumption \ref{Assumption:3} (iv') we assume that $Db$ and $D\sigma$ are Lipschitz in the spacial variable. We chose this condition because it is easy to verify and strong enough to ensure that $\forall x\in \bR^d$ 
\begin{align*}
\bE\Big[ \Big( \int_0^T \Big( \int_0^t |D_sb(t, \omega, X(t))|^2 ds \Big)^{\tfrac{1}{2}} dt \Big)^p\Big] < \infty, \quad
\bE\Big[ \Big( \int_0^T \int_0^t |D_s\sigma(t, \omega, X(t))|^2 ds dt \Big)^{\tfrac{p}{2}} \Big] < \infty.
\end{align*}
However, this condition is by no means necessary. One could consider the case where $Db$ is locally Lipschitz in space and satisfies a linear growth condition and equivalently prove Theorem \ref{theo:WhatsHere}. However, the proof is more involved as it involves a careful interplay using H\"older's inequality between the maximal integrability of $X$, $Db$, $D\sigma$ and several other stochastic terms. 
 
\end{remark}

\subsection[Proof of the 1st main result]{Proofs of the 1st main result - Theorem \ref{theo-Mall-diff-monotone-SDEs}}
\label{sec:Proof of 1st MallDiff main theorem}

In what follows, the choice of $\theta$ (the initial condition in \eqref{eq:SDE}) does not affect the Malliavin derivative because $\theta$ is $\cF_0$-measurable. If $Y$ is $\cF_t$-measurable then $D_s Y=0$ for any $t<s$, see \cite{nualart2006malliavin}*{Corollary 1.2.1}.

\subsubsection*{Existence and Uniqueness of the Malliavin derivative $D_sX(t)$}
We start by establishing that \eqref{eq:SDEMallDeriv} has a unique solution where $X$ solves \eqref{eq:SDE}. At this point, nothing is said about the solution of \eqref{eq:SDEMallDeriv} being the Malliavin derivative to $X$ solution of \eqref{eq:SDE}, showing it is the subsequent step.
\begin{theorem}
\label{theo:ExistenceUniqueness for Candidate DX}
Let $p>2$. For $(s,t)\in [0,T]^2$, let $X$ be the solution to the SDE \eqref{eq:SDE} under Assumption \ref{Assumption:2}. Let $(M_s(t))$ be defined by the matrix of $L^2([0,T])$-valued SDEs
\begin{align}
\label{eq:SDEMallCandidate}
M_s(t)(\omega) =& \sigma(s,\omega, X(s)(\omega)) +  \int_s^t U(s, r, \omega) dr + \int_s^t V(s, r, \omega) dW(r)
\\
\nonumber
&+\int_s^t \nabla_x b(r,\omega, X(r)(\omega)) M_s(r)(\omega) dr + \int_s^t \nabla_x \sigma(r,\omega, X(r)(\omega)) M_s (r)(\omega) dW(r),
\end{align}
for $s<t$ and $M_s(t)=0$ for $s>t$. 

Then a unique solution exists in $\cS^p([0,T]; L^2([0,T]))$ for \eqref{eq:SDEMallCandidate} and the process $M$ has finite $p^{th}$ moment, namely 
\begin{align*}
\bE\Big[ \Big(\sup_{t\in[0,T]} \int_0^T |M_s (t)|^2 ds\Big)^{\tfrac{p}{2}}\Big] < \infty. 
\end{align*}
\end{theorem}
Observe that Equation \eqref{eq:SDEMallCandidate} is linear in $M$, so the sharpness of the integrability is determined by the integrability of $U$, $V$ and $\sigma$ (given the assumed behavior of $\nabla_x b$ and $\nabla_x \sigma$). In the trivial case where $U = V = 0$ and $\sigma=1$ then $M$ has finite moments of all orders.

\begin{proof}[Proof of Theorem \ref{theo:ExistenceUniqueness for Candidate DX}]
For brevity, $t\in[0,T]$ and we omit the explicit $\omega$ dependency throughout.  

Equation \eqref{eq:SDEMallCandidate} is an infinite dimensional SDE. We see this when we think of the Malliavin Derivative as being an $L^2([0,T])$ valued stochastic process. Therefore, we need to extend results from Section 2 to infinite dimensional spaces. Let $e_n$ be an orthonormal basis of the space $L^2([0,T]; \bR^{m})$. This is a separable Hilbert space, so without loss of generality we can say the orthonormal basis is countably infinite.  Let $V_n$ be the linear span of the set $\{ e_1, ..., e_n\}$. Let $P_n : L^2([0,T]; \bR^{m}) \to V_n$ be the canonical projection operators
\begin{align*}
P_n[f](t) = \sum_{k=1}^n \langle f, e_k \rangle_{L^2([0,T]; \bR^{m})} e_k(t).
\end{align*}
Then it is clear that $\lim_{n\to \infty} \| P_n[f] - f\|_{L^2([0,T]; \bR^{ m})} = 0$.  For $k\in \bN$, consider the sequence of 1-dimensional Linear Stochastic Differential Equations

\begin{align*}
M_k (t) =& \int_0^t \sigma(u, X(u)) e_k(u) du + \int_0^t \Big( \int_0^r U(u, r) e_k(u) du\Big) dr + \int_0^t \Big( \int_0^r V(u, r) e_k(u) du \Big) dW(r)
\\
&+\int_0^t \nabla_x b(r, X(r)) M_k(r) dr + \int_0^t \nabla_x \sigma (r,  X(r)) M_k(r) dW(r).
\end{align*}
These equations are of the same form as \eqref{eq:SDEDeriv}, hence a unique solution exists for each $k$ by Theorem \ref{theorem:ExistenceUniquenessMaoLinearSDE}. Also, observe that the fundamental matrix $\Psi$ will be the same for each choice of $k\in \bN$. $\Psi$ will have the explicit solution
$$
\Psi(t) = \exp\Big( \int_0^t \nabla b(r, X(r)) dr - \frac{1}{2}\int_0^t \Big\langle \nabla \sigma\big(r, X(r)\big), \nabla \sigma\big(r, X(r)\big) \Big\rangle_{\bR^m} dr + \int_0^t \nabla \sigma(r, X(r)) dW(r) \Big)
$$
and $M_k$ has explicit solution
\begin{align*}
M_k(t) =&\Psi(t) \Bigg( \int_0^t \sigma\big(u, X(u)\big) e_k(u) du + \int_0^t \Psi(r)^{-1} \Big[ \int_0^r U\big(u, r\big) e_k(u) du 
\\
&- \Big\langle \nabla\sigma\big(r, X(r)\big) , \int_0^r V\big(u, r\big) e_k(u) du \Big\rangle_{\bR^m} \Big] dr +\int_0^t \Psi(r)^{-1} \int_0^r V\big(u, r\big) e_k(u) du dW(r) \Bigg). 
\end{align*}

Next define for $0\leq s,t\leq T$, $n\in\bN$ the process
\begin{align*}
M_{(n), s}(t) = \sum_{k=1}^n M_k(t) \otimes e_k(s) \1_{[0,t)} (s). 
\end{align*}

This process makes sense as the projection space is finite dimensional so we can rewrite it in a finite dimensional vector form. The solution exists in the space $\cS^p(L^2([0,T]; \bR^{d\times m}))$ and has the explicit solution
\begin{align*}
M_{(n), s}(t) =& \sum_{k = 1}^n \Psi(t) \Bigg( \int_0^t \sigma\big(u, X(u)\big) e_k(u) du + \int_0^t \Psi(r)^{-1} \Big[ \int_0^r U\big(u, r\big) e_k(u) du 
\\
&- \Big\langle \nabla\sigma\big(r, X(r)\big) , \int_0^r V\big(u, r\big) e_k(u) du \Big\rangle_{\bR^m} \Big] dr 
\\
&+\int_0^t \Psi(r)^{-1} \Big[\int_0^r V\big(u, r\big) e_k(u) du \Big] dW(r) \Bigg) \otimes e_k(s)  \1_{[0,t)} (s),
\\
=& \Psi(t)\Bigg( P_n\Big[ \sigma(\cdot, X(\cdot))\Big] (s) + \int_s^t  \Psi(r)^{-1} \Big( P_n\Big[ U(\cdot, r)\Big](s)
\\
&- \Big\langle \nabla\sigma\big(r, X(r)\big), P_n\Big[V\big(\cdot, r\big)\Big](s) \Big\rangle_{\bR^m} \Big) dr + \int_s^t \Psi(r)^{-1} P_n\Big[ V\big(\cdot, r\big) \Big](s) dW(r) \Bigg). 
\end{align*}

This process satisfies the SDE
\begin{align*}
M_{(n), s}(t) =& P_n\Big[ \sigma(\cdot, X(\cdot)) \Big](s) + \int_s^t P_n\Big[ U(\cdot, r) \Big](s) dr + \int_s^t P_n\Big[ V(\cdot, r) \Big](s) dW(r)
\\
&+\int_s^t \nabla_x b(r, X(r)) M_{(n), s}(r) dr + \int_s^t \nabla_x \sigma(r, X(r)) M_{(n), s} (r) dW(r).
\end{align*}

We require a norm on the space of $L^2$-valued matrices. Let $a^{(i, j)}\in L^2([0,T])$ and for $A(u)=(a^{(i, j)}(u))_{i\in\{1, ..., d\}, j\in\{1, ..., m\}}$, define

\begin{align}
\label{eq:Matrix-norm-of-convinience}
\|A_{\cdot}\| = \Big( \sum_{i=1}^d \sum_{j=1}^m \int_0^T |a^{(i, j)}(u)|^2 du \Big)^{1/2} = \Big( \int_0^T |A(u)|^2 du\Big)^{1/2}.
\end{align}

By It\^o's formula, we have
\begin{align*}
\Big|&M_{(n), s}(t) - M_{(m), s}(t)\Big|^2= \sum_{i=1}^d \sum_{k=1}^{m} \Big|M_{(n), s}^{(i, k)}(t) - M_{(m), s}^{(i, k)}(t)\Big|^2, 
\\
&=\sum_{i, k}\Big|(P_n - P_m)\Big[ \sigma(\cdot, X(\cdot)) \Big]^{(i, k)}(s) \Big|^2
\\
&+ 2\sum_{i, k}\int_s^t \Big(M_{(n), s}^{(i, k)}(r) - M_{(m), s}^{(i, k)}(r)\Big)\cdot (P_n - P_m)\Big[ U(\cdot, r) \Big]^{(i, k)}(s) dr 
\\
&+ 2\sum_{i, j, k}\int_s^t \Big(M_{(n), s}^{(i, k)}(r) - M_{(m), s}^{(i, k)}(r)\Big)\cdot (P_n - P_m)\Big[ V(\cdot, r) \Big]^{(i, j, k)}(s) dW^{(j)}(r) 
\\
&+ 2\sum_{i, k}\int_s^t \Big(M_{(n), s}^{(i, k)}(r) - M_{(m), s}^{(i, k)}(r)\Big)\cdot\Big\langle \nabla_x b^{(i)}(r, X(r)), M_{(n), s}^{(\cdot, k)}(r) - M_{(m), s}^{(\cdot, k)}(r) \Big\rangle dr 
\\
&+ 2\sum_{i, j, k}\int_s^t \Big(M_{(n), s}^{(i, k)}(r) - M_{(m), s}^{(i, k)}(r)\Big)\cdot\Big\langle \nabla_x \sigma^{(i ,j)} (r, X(r)), M_{(n), s}^{(\cdot, k)} (r) - M_{(m), s}^{(\cdot, k)} (r)\Big\rangle dW^{(j)}(r)
\\
&+ \sum_{i, j, k}\int_s^t \Big| (P_n - P_m)\Big[ V(\cdot, r) \Big]^{(i, j, k)}(s) + \Big\langle \nabla_x \sigma^{(i, j)}(r, X(r)), M_{(n), s}^{(\cdot, k)} (r) - M_{(m), s}^{(\cdot, k)} (r) \Big|^2 dr.
\end{align*}

Denote $N_s(t) = M_{(n), s}(t) - M_{(m), s}(t)$ and $(P_n - P_m) = Q$ for brevity. Integrating over $s$ and since every term is positive, we can change the order of integration to obtain
\begin{align*}
\int_0^t \Big|&N_s^{(i, k)}(t)\Big|^2ds = \int_0^t \Big|Q\big[ \sigma(\cdot, \omega, X(\cdot)) \big]^{(i, k)}(s) \Big|^2 ds
\\
&+ 2\int_0^t\int_0^r N_s^{(i, k)}(r) \cdot \Big[Q\big[ U(\cdot, r) \big]^{(i, k)}(s) + \Big\langle \nabla_x b^{(i)}(r, X(r)), N_s^{(\cdot, k)}(r)\Big\rangle \Big] dsdr 
\\
&+ 2\sum_{j}\int_0^t\int_0^r N_s^{(i, k)}(r) \cdot \Big[Q\big[ V(\cdot, r) \big]^{(i, j, k)}(s) + \Big\langle \nabla_x \sigma^{(i, j)}(r, X(r)), N_s^{(\cdot, k)}(r)\Big\rangle \Big] ds dW^{(j)}(r) 
\\
&+ \sum_{j}\int_0^t\int_0^r \Big|Q\big[ V(\cdot, r) \big]^{(i, j, k)}(s) + \Big\langle \nabla_x \sigma^{(i, j)}(r, X(r)), N_s^{(\cdot, k)}(r)\Big\rangle \Big|^2 dsdr. 
\end{align*}

Next, we use It\^o's formula $g(x) = \big(\sum_i x^{(i)} \big)^{{p}/{2}}$ to get

\begin{align}
\nonumber
\| &N_\cdot(t) \|^{p} = \Big( \sum_{i, k} \int_0^t | N_s^{(i, k)}(t)|^2 ds \Big)^{\tfrac{p}{2}}
= \Big( \int_0^t \Big|Q\big[ \sigma(\cdot, \omega, X(\cdot)) \big](s) \Big|^2 ds \Big)^{\tfrac{p}{2}}
\\
\label{eq:MalliavinSequence1}
&+p\int_0^t \| N_\cdot(r) \|^{p-2} \Big( \sum_{i, k}\int_0^r N_s^{(i, k)}(r) \Big[ Q\big[ U(\cdot, r) \big]^{(i, k)}(s) + \big\langle \nabla_x b^{(i)}\big(r, X(r)\big),  N_s^{(\cdot, k)}(r)\Big\rangle ds\Big) dr 
\\
\label{eq:MalliavinSequence2}
&+\frac{p}{2}\int_0^t \| N_\cdot(r) \|^{p-2} \sum_{i, j, k} \int_0^r \Big|Q\big[ V(\cdot, r) \big]^{i, j, k}(s) + \Big\langle\nabla_x \sigma^{(i, j)}(r, X(r)), N_s^{(\cdot, k)}(r)\Big\rangle \Big|^2 ds dr 
\\
\label{eq:MalliavinSequence3}
&+p\int_0^t \| N_\cdot(r) \|^{p-2} \sum_{i, j, k} \int_0^r N_s^{(i, k)}(r) \Big( Q\big[ V(\cdot, r) \big]^{(i, j, k)}(s) + \Big\langle \nabla_x \sigma^{(i, j)}\big(r, X(r)\big), N_s^{(\cdot, k)}(r) \Big) ds dW(r) 
\\
\nonumber
\label{eq:MalliavinSequence4}
&
+p(p-2)\int_0^t  \| N_\cdot(r) \|^{p-4} \sum_{i, j, k} \Big( \int_0^r N_s^{(i, k)}(r)  \Big( Q\big[ V(\cdot, r) \big]^{(i, j, k)}(s)
\\
& \hspace{8cm}
 + \Big\langle \nabla_x \sigma^{(i, j)}\big(r, X(r)\big), N_s^{(\cdot, k)}(r) \Big) ds\Big)^2 dr . 
\end{align}

We take a supremum over $t\in[0,T]$ then expectations to show that $\bE[ \| N \|_\infty^2]< \varepsilon$ for $n, m\in \bN$ large enough. Let $a\in \bN$ be an integer which we will choose later.

Firstly, 
\begin{align}
\label{eq:MalliavinSequence1.1}
\eqref{eq:MalliavinSequence1} \leq p&\bE\Big[ \int_0^T \| N_\cdot(r)\|^{p-2} \Big( \sum_{i, k} \int_0^r N_s^{(i, k)} (r) Q\big[ U(\cdot, r)\big]^{(i, k)}(s) ds\Big) dr\Big]
\\
\label{eq:MalliavinSequence1.2}
&+p\bE\Big[ \int_0^T \| N_\cdot(r)\|^{p-2}\Big( \sum_{i, k} \int_0^r N_s^{(i, k)}(r) \Big\langle \nabla_x b^{(i)}\big(r, X(r)\big),  N_s^{(\cdot, k)}(r)\Big\rangle ds\Big) dr\Big]. 
\end{align}
Now we deal with \eqref{eq:MalliavinSequence1.1} using H\"older inequality, the norm \eqref{eq:Matrix-norm-of-convinience}, then dominate via the supremum norm and move the term outside the integral to merge it with the outer integrand term
\begin{align*}
\eqref{eq:MalliavinSequence1.1} \leq& p\bE\Big[ \|N_\cdot\|_\infty^{p-1} \int_0^T \Big( \sum_{i, k}\int_0^r | Q[U(\cdot, r)]^{(i, k)}(s) |^2 ds \Big)^{\tfrac{1}{2}} dr \Big]
\\
\leq& \frac{\bE\Big[ \| N_\cdot \|_\infty^p\Big]}{a} + [a(p-1)]^{p-1} \bE\Big[ \Big( \int_0^T \| Q[U(\cdot, r)](s)\|_2 ds \Big)^p \Big],
\end{align*}
and 
$$
\eqref{eq:MalliavinSequence1.2} \leq pL \int_0^T \bE\Big[ \|N_\cdot \|_{\infty, r}^p\Big] dr.
$$
using the Monotonicity property of $b$. Secondly, 
\begin{align*}
\eqref{eq:MalliavinSequence2} \leq & p \bE \Big[ \int_0^T \| N_{\cdot}(r) \|^{p-2} \Big( \sum_{i, j, k} \int_0^r |Q[V(\cdot, r)]^{(i, j, k)}|^2 ds \Big) dr\Big] 
\\ 
&
+ p \bE \Big[ \int_0^T \|N_{\cdot}(r) \|^{p-2} \Big( \sum_{i, j, k} \int_0^r \Big\langle \nabla_x \sigma^{(i, j)}(r, X(r)), N^{(\cdot, k)}_s(r) \Big\rangle^2ds \Big) dr\Big] 
\\
\leq& \frac{\bE\Big[ \| N_{\cdot} \|_{\infty}^{p}\Big]}{a} + 2[a(p-2)]^{\tfrac{p-2}{2}} \bE\Big[ \Big( \int_0^T \| Q[V(\cdot, r)](\cdot)\|_2^2 dr \Big)^{\tfrac{p}{2}} \Big]
+  pL^2 \int_0^t \bE \Big[ \| N_{\cdot} \|_{\infty, r}^{p}\Big]dr,
\end{align*}
using the boundedness of $\nabla \sigma$. Thirdly, using the Burkholder-Davis-Gundy Inequality
\begin{align}
\nonumber
\eqref{eq:MalliavinSequence3}\leq 
& pC_1\bE\Big[ \Big( \int_0^{T} \| N_{\cdot}(r) \|^{2p-4} \sum_{j} \Big( \sum_{i, k} \int_0^r N^{(i, k)}_s(r) \Big[ Q[V(\cdot, r)]^{(i, j, k)}(s)
\\
\nonumber
&+ \Big\langle \nabla_x \sigma^{(i, j)}(r, X(r)), N^{(\cdot, k)}_s(r) \Big\rangle\Big] ds\Big)^2 dr\Big)^{\tfrac{1}{2}}\Big] 
\\
\label{eq:MalliavinSequence3.1}
\leq&\sqrt{2}pC_1 \bE\Big[ \| N_{\cdot} \|_{\infty}^{p-2} \Big( \int_0^T \Big[ \sum_{i, j, k} \int_0^r |N^{(i, k)}_s(r)| \cdot | Q[V(\cdot, r)]^{(i, j, k)}(s) |ds \Big]^2 dr \Big)^{\tfrac{1}{2}}\Big] 
\\
\label{eq:MalliavinSequence3.2}
&+\sqrt{2}pC_1 \bE\Big[ \| N_{\cdot} \|_{\infty}^{p-2} \Big( \int_0^T \Big[ \sum_{i, j, k} \int_0^r |N^{(i, k)}_s(r)| \cdot \Big| \big\langle \nabla_x \sigma^{(i, j)}(r, X(r)), N^{(\cdot, k)}_s(r) \big\rangle\Big|ds \Big]^2 dr \Big)^{\tfrac{1}{2}}\Big] . 
\end{align}

As before, we have
\begin{align*}
\eqref{eq:MalliavinSequence3.1}\leq& \sqrt{2} pC_1 \bE \Big[ \| N_{\cdot}\|_{\infty}^{p-1} \Big( \int_0^T \|Q[V(\cdot, r)](\cdot)\|_2^2dr\Big)^{\tfrac{1}{2}}\Big] 
\\
\leq& \frac{\bE\Big[ \| N_{\cdot} \|_{\infty}^{p} \Big]}{a} + (\sqrt{2} C_1)^p[a(p-1)]^{p-1} \bE \Big[ \Big( \int_0^T\|Q[V(\cdot, r)](\cdot)\|_2^2 dr\Big)^{\tfrac{p}{2}}\Big],
\end{align*}
and put together
\begin{align*}
\eqref{eq:MalliavinSequence3.2}\leq& \sqrt{2} pC_1 L \bE\Big[ \Big( \int_0^T \| N_{\cdot}(r)\|^{2p} dr\Big)^{\tfrac{1}{2}}\Big] 
\leq \sqrt{2} pC_1 L \bE\Big[ \| N_\cdot\|_{\infty}^{\tfrac{p}{2}} \Big( \int_0^T \| N_\cdot(r)\|^{p} dr\Big)^{\tfrac{1}{2}}\Big] 
\\
\leq& \frac{\bE\Big[ \| N_{\cdot}\|_{\infty}^{p} \Big]}{a} + \frac{(pC_1L)^2 a}{2}\int_0^T \bE\Big[ \| N_{\cdot}\|_{\infty, r}^p\Big] dr. 
\end{align*}

Finally, for 
\begin{align}
\label{eq:MalliavinSequence4.1}
\eqref{eq:MalliavinSequence4}\leq& 2p(p-2) \bE\Big[ \int_0^T \|N_{\cdot}(r)\|^{p-4} \Big( \sum_{i, j, k} \int_0^r |N^{(i, k)}_s(r)|\cdot | Q[V(\cdot, r)]^{(i, j, k)}(s) |ds \Big)^2dr \Big]
\\
\label{eq:MalliavinSequence4.2}
&+2p(p-2) \bE\Big[ \int_0^T \| N_{\cdot}(r)\|^{p-4} \Big( \sum_{i, j, k} \int_0^r |N^{(i, k)}_s(r)|\cdot \Big|\Big\langle \nabla_x \sigma^{(i, j)}(r, X(r)), N^{(\cdot, k)}_s(r) \Big\rangle\Big|ds \Big)^2dr \Big].
\end{align}

Repeating the same ideas as before, we get
\begin{align*}
\eqref{eq:MalliavinSequence4.1} \leq& 2p(p-2) \bE\Big[ \|N_{\cdot}\|_{\infty}^{p-2} \cdot \int_0^T\| Q[V(\cdot, r)](\cdot)\|_2^2 dr\Big] 
\\
\leq&\frac{\bE\Big[ \| N_{\cdot}\|_{\infty}^{p} \Big]}{a} + 2^{\tfrac{p+2}{2}}\cdot a^{\tfrac{p-2}{2}}\cdot (p-2)^{p-1}\bE\Big[ \Big( \int_0^T \| Q[V(\cdot, r)](\cdot)\|_2^2 dr\Big)^{\tfrac{p}{2}}\Big].
\end{align*}
and $\eqref{eq:MalliavinSequence4.2} \leq 2p(p-2)L^2 \bE\big[ \int_0^T \| N_{\cdot}\|_{\infty, r}^{p} dr \big]$. Therefore, choosing $a=6$ we conclude
\begin{align*}
\frac{\bE\Big[ \| N_{\cdot}\|_{\infty}^p \Big]}{6} 
\leq& \Bigg( \bE\Big[ \| Q[\sigma(\cdot, X(\cdot))](\cdot)\|_2^{p} \Big] 
+ \widetilde{C}_1 \bE\Big[ \Big( \int_0^T \|Q[U(\cdot, r)](\cdot)\|_2dr\Big)^{p} \Big] 
\\
&+\widetilde{C}_2 \bE\Big[ \Big( \int_0^T \|Q[V(\cdot, r)](\cdot)\|_2^2dr\Big)^{\tfrac{p}{2}} \Big] \Bigg) + \widetilde{C}_3 \int_0^T \bE\Big[ \| N_{\cdot}\|_{\infty, r}^p \Big] dr.
\end{align*}
By applying Gr\"onwall's inequality we conclude that 
\begin{align*}
\bE\Big[ \sup_{t\in[0,T]} \| &M_{(n), \cdot} - M_{(m), \cdot} \|^p \Big]\lesssim \Bigg( \bE\Big[ \| (P_n - P_m)\big[\sigma(\cdot, X(\cdot))\big]\|_2^{p} \Big] 
\\
&+ \bE\Big[ \Big( \int_0^T \|(P_n - P_m)[U(\cdot, r)]\|_2 dr\Big)^{p} \Big] +\bE\Big[ \Big( \int_0^T \|(P_n - P_m)[V(\cdot, r)]\|_2^2dr\Big)^{\tfrac{p}{2}} \Big] \Bigg). 
\end{align*}

Given that, by Assumption we already have
$$
\bE\Big[ \| \sigma(\cdot, X(\cdot))\|_2^{p} \Big], \quad \bE\Big[ \Big( \int_0^T \|U(\cdot, r)\|_2 dr\Big)^{p} \Big], \quad \bE\Big[ \Big( \int_0^T \|V(\cdot, r)\|_2^2dr\Big)^{\tfrac{p}{2}} \Big] < \infty, 
$$
we are able to apply the Dominated Convergence Theorem to swap the order of limits and integrals. Taking a limit as $m, n$ go to infinity lets us conclude that the sequence $M_{(n)}$ is Cauchy in $\cS^p(L^2([0,T]; \bR^{d\times m}))$. This is a Banach space, so a limit must exist which we denote by $M'$,
\begin{align*}
M'_s(t)=& \lim_{n\to\infty} \Psi(t)\Bigg( P_n\Big[ \sigma(\cdot, X(\cdot))\Big] (s) + \int_s^t  \Psi(r)^{-1} \Big( P_n\Big[ U(\cdot, r)\Big](s)
\\
&- \Big\langle \nabla\sigma\big(r, X(r)\big), P_n\Big[V\big(\cdot, r\big)\Big](s) \Big\rangle_{\bR^m} \Big) dr + \int_s^t \Psi(r)^{-1} P_n\Big[ V\big(\cdot, r\big) \Big](s) dW(r) \Bigg).
\end{align*}

Now let $g\in L^2([0,T]; \bR^m)$ be chosen arbitrarily. Then we define ${M^g}' (\cdot)$ as
\begin{align*}
{M^g}'(t)=& \int_0^t M'_s(t)g(s)ds
\\=& \Psi(t)\Bigg( \int_0^t \sigma(s, X(s))g(s) ds + \int_s^t  \Psi(r)^{-1} \Big( \int_0^t U(s, r) g(s)ds
\\
&- \Big\langle \nabla\sigma\big(r, X(r)\big), \int_0^t V\big(s, r\big) g(s)ds\Big\rangle_{\bR^m} \Big) dr + \int_s^t \Psi(r)^{-1} \int_0^t V\big(s, r\big)g(s)ds dW(r) \Bigg). 
\end{align*}
In order to move the limit inside the different integrals, we use the Dominated Convergence Theorem again. 

Given an explicit solution, we know ${M^g}'$ will satisfy the SDE
\begin{align*}
{M^g}' (t) =& \int_0^t \sigma(s, X(s)) g(s) ds + \int_0^t \Big( \int_0^r U(s, r) g(s) ds\Big) dr + \int_0^t \Big( \int_0^r V(s, r) g(s) ds \Big) dW(r)
\\
&+\int_0^t \nabla_x b(r, X(r)) {M^g}'(r) dr + \int_0^t \nabla_x \sigma (r, X(r)) {M^g}'(r) dW(r).
\end{align*}
Therefore by a duality argument 
\begin{align*}
M'_s (t) =& \sigma(s, X(s)) + \int_0^t U(s, r) dr + \int_0^t V(s, r) dW(r)
\\
&+\int_0^t \nabla_x b(r, X(r)) M'_s(r) dr + \int_0^t \nabla_x \sigma (r, X(r)) M'_s(r) dW(r),
\end{align*}
which is the same SDE as \eqref{eq:SDEMallCandidate}. 

Next we prove uniqueness. Suppose that there are two solutions to the SDE \eqref{eq:SDEMallCandidate}, $M$ and $M'$. Denote $M - M' = \tilde{N}$. Then $\tilde{N}$ will satisfy the linear SDE 

\begin{align*}
d\tilde{N}_s(t) = \nabla_x b(t, X(t)) \tilde{N}_s(t) dt +  \nabla_x \sigma (t, X(t)) \tilde{N}_s(t) dW(t), \quad \tilde{N}_s(s) = 0. 
\end{align*}

Let $g\in L^2([0,T]; \bR^m)$ be chosen arbitrarily. Define $\tilde{N}^g(t) = \int_0^t \tilde{N}_s(t) g(s) ds$.  
Clearly, this linear SDE will almost surely be equal to 0 independently of the choice of $g$. Hence $\tilde{N}$ must also be equal to 0. So $M = M'$ and we have proved uniqueness. 

\end{proof}

\subsubsection*{Ray Absolute Continuity of $X$}

We show that the expectation of $\| \big( {X(\cdot)(\omega+\varepsilon h) - X(\cdot)(\omega)}\big)/{\varepsilon} \|_\infty^2$ has a bound uniform in $\varepsilon$. This relies on having finite $p^{th}$ moments of the random variable $\|X\|_{\infty}$ for $p>2$.  

The case $p=2$ is problematic. It is not the case that $Z\in\bD^{1, 2}$ implies that $\big(Z(\omega+\varepsilon h) - Z(\omega)\big)/{\varepsilon}$ converges in mean square as $\varepsilon \searrow 0$, see Remark \ref{remark:SquareIntegrabilityProblems} and \cite{imkeller2016note} for in-depth discussion. If we were dealing with the sharp case where the solution of the SDE exists in $\cS^2$, it would be unreasonable to expect the Malliavin Derivatives of $b$ and $\sigma$ to satisfy Assumption \ref{Assumption:2}(iv), which is necessary for the following Proposition. The power $p$ must be greater that $2$, as opposed to $1$, because the monotonicity condition lends itself to studying the moments of the SDE for moments of greater than or equal to $2$ but is a hindrance for the moments of order less than $2$ (computations may involve local times).

\begin{proposition}
\label{proposition:MallCalcRayAbsoluteCont}
Let $X$ be solution to the SDE \eqref{eq:SDE} under Assumption \ref{Assumption:2}. We have
\begin{equation}
\label{eq:RACviaMean}
\bE\Big[\, \Big\| \frac{X(\cdot)(\omega+\varepsilon h) - X(\cdot)(\omega)}{\varepsilon}\Big\|_\infty^2\Big] =O(1) \quad \textrm{as }\varepsilon\to 0.
\end{equation} 
\end{proposition}
After we have proved Stochastic G\^ateaux Differentiability (see Theorem \ref{theo:GateauxDifferentiabilitySDE}), Corollary \ref{corollary:HowToRAC} and Equation \eqref{eq:RACviaMean} will imply Ray Absolute Continuity. 

\begin{proof}
Let $t \in [0,T]$. Using  Assumption \ref{Assumption:2}, we have 
\begin{align*}
\bE\Big[& \Big( \int_0^T \Big| \frac{b(t, \omega+\varepsilon h, X(t)(\omega)) - b(t, \omega, X(t)(\omega))}{\varepsilon} \Big| dt\Big)^2 \Big] 
\\
&
\leq 
2 \bE\Bigg[ \|h\|_2^2 \Big( \int_0^T \Big(\int_0^t |U(s, t, \omega)|^2 ds\Big)^{\tfrac{1}{2}} dt\Big)^2 
\\
&\qquad + \Big( \int_0^T \Big| \frac{b(t, \omega + \varepsilon h, X(t)) - b(t, \omega, X(t))}{\varepsilon} - \int_0^t U(s, t, \omega)\dot{h}(s)ds\Big| dt\Big)^2 \Bigg] 
\leq O(1), 
\end{align*}
and
\begin{align*}
\bE\Big[& \int_0^T \Big| \frac{\sigma(t, \omega+\varepsilon h, X(t)(\omega)) - \sigma(t, \omega, X(t)(\omega))}{\varepsilon}\Big|^2 ds\Big]
\\
&
\leq 
2 \bE\Bigg[ \|h\|_2^2 \int_0^T \int_0^t |V(s, t, \omega)|^2 ds dt
\\
&\qquad + \int_0^T \Big| \frac{\sigma(t, \omega + \varepsilon h, X(t)) - \sigma(t, \omega, X(t))}{\varepsilon} - \int_0^t V(s, t, \omega)\dot{h}(s)ds\Big|^2 dt \Bigg] 
\leq O(1). 
\end{align*} 

For notational compactness let us introduce $P_\varepsilon(t)(\omega) = \big({X(t)(\omega+\varepsilon h) - X(t)(\omega)}\big)/{\varepsilon}$. We have   
\begin{align*}
P_\varepsilon(t)(\omega) 
=& \int_0^t \sigma(s, \omega, X(s)(\omega)) \dot{h}(s) ds
\\
&+ \int_0^t \Big( \sigma(s, \omega+\varepsilon h, X(s)(\omega+\varepsilon h)) - \sigma(s, \omega, X(s)(\omega)) \Big)\dot{h}(s) ds
\\
&+ \frac{1}{\varepsilon} \int_0^t \Big( b(s, \omega+\varepsilon h, X(s)(\omega+\varepsilon h)) - b(s, \omega, X(s)(\omega))\Big) ds
\\
&+ \frac{1}{\varepsilon} \int_0^t \Big( \sigma(s, \omega+\varepsilon h, X(s)(\omega+\varepsilon h)) - \sigma(s, \omega, X(s)(\omega))\Big) dW(s).
\end{align*}
Using It\^o's formula for $f(x) = x^2$ we have
\begin{align}
\label{eq:P(t),1}
\Big| &P_\varepsilon(t)(\omega) \Big|^2 = 2\int_0^t \Big\langle P_\varepsilon(s)(\omega), \sigma(s, \omega, X(s)(\omega)) \dot{h}(s)\Big\rangle ds
\\
\label{eq:P(t),2}
&+ 2\int_0^t \Big\langle P_\varepsilon(s)(\omega) ,\Big( \sigma(s, \omega+\varepsilon h, X(s)(\omega+\varepsilon h)) - \sigma(s, \omega+\varepsilon h, X(s)(\omega)) \Big)\dot{h}(s) \Big\rangle ds
\\
\label{eq:P(t),3}
&+ 2\int_0^t \Big\langle P_\varepsilon(s)(\omega) ,\Big( \sigma(s, \omega+\varepsilon h, X(s)(\omega)) - \sigma(s, \omega, X(s)(\omega)) \Big)\dot{h}(s) \Big\rangle ds
\\
\label{eq:P(t),4}
&+ 2\int_0^t \Big\langle P_\varepsilon(s)(\omega), \frac{b(s, \omega+\varepsilon h, X(s)(\omega+\varepsilon h)) - b(s, \omega+\varepsilon h, X(s)(\omega))}{\varepsilon} \Big\rangle ds
\\
\label{eq:P(t),5}
&+ 2\int_0^t \Big\langle P_\varepsilon(s)(\omega), \frac{b(s, \omega+\varepsilon h, X(s)(\omega)) - b(s, \omega, X(s)(\omega))}{\varepsilon} \Big\rangle ds
\\
\label{eq:P(t),6}
&+ 2\int_0^t \Big\langle P_\varepsilon(s)(\omega), \frac{\sigma(s, \omega+\varepsilon h, X(s)(\omega+\varepsilon h)) - \sigma(s, \omega+\varepsilon h, X(s)(\omega))}{\varepsilon} dW(s) \Big\rangle 
\\
\label{eq:P(t),7}
&+ 2\int_0^t \Big\langle P_\varepsilon(s)(\omega), \frac{\sigma(s, \omega+\varepsilon h, X(s)(\omega)) - \sigma(s, \omega, X(s)(\omega))}{\varepsilon} dW(s)\Big\rangle 
\\
\label{eq:P(t),8}
&+ \int_0^t \Big| \frac{\sigma(s, \omega+\varepsilon h, X(s)(\omega+\varepsilon h)) - \sigma(s, \omega, X(s)(\omega))}{\varepsilon} \Big|^2 ds. 
\end{align}

We take a supremum over $t$ then expectations. Let $n$ be an integer that we will choose later. By using a combination of Young's Inequality, Cauchy-Schwartz Inequality,  Burkholder-Davis-Gundy Inequality and the continuity properties from  Assumption \ref{Assumption:2} we find the following upper bounds:
\begin{align*}
\textrm{For }\eqref{eq:P(t),1} \Rightarrow 
&\bE\Big[ 2\int_0^T \Big| \Big\langle P_\varepsilon(s)(\omega), \sigma(s, \omega, X(s)(\omega)) \dot{h}(s)\Big\rangle\Big| ds\Big]
\\
\leq&\frac{\bE[ \| P_\varepsilon\|_{\infty}^2]}{n} + n\|\dot{h}\|_2^2 \bE\Big[ \int_0^T \Big| \sigma(s, \omega, X(s)(\omega))\Big|^2 ds\Big] 
\\
\leq&\frac{\bE[ \| P_\varepsilon\|_{\infty}^2]}{n} + 2n\|\dot{h}\|_2^2 \Big( L^2 \bE\Big[ \|X\|_{\infty}^2 \Big] + \bE\Big[\int_0^T \Big|\sigma(s, \omega, 0)\Big|^2 ds \Big]\Big),
\end{align*}
\begin{align*}
\textrm{For }\eqref{eq:P(t),2} \Rightarrow&\bE\Big[ 2\int_0^T \Big| \Big\langle P_\varepsilon(s)(\omega) ,\Big( \sigma(s, \omega+\varepsilon h, X(s)(\omega+\varepsilon h)) - \sigma(s, \omega+\varepsilon h, X(s)(\omega)) \Big)\dot{h}(s) \Big\rangle \Big| ds\Big]
\\
\leq& 2L \varepsilon \int_0^T \bE\Big[\| P_\varepsilon\|_{\infty, s}^2 \Big] \cdot|\dot{h}(s)| ds,
\end{align*}
\begin{align*}
\textrm{For }\eqref{eq:P(t),3} \Rightarrow&\bE\Big[ 2\int_0^T \Big| \Big\langle P_\varepsilon(s)(\omega) ,\Big( \sigma(s, \omega+\varepsilon h, X(s)(\omega)) - \sigma(s, \omega, X(s)(\omega)) \Big)\dot{h}(s) \Big\rangle\Big| ds \Big]
\\
\leq& \frac{\bE[ \| P_\varepsilon\|_{\infty}^2]}{n} + n\|\varepsilon \dot{h}\|_2^2 \bE\Big[ \int_0^T \Big|  \frac{\sigma(s, \omega+\varepsilon h, X(s)(\omega)) - \sigma(s, \omega, X(s)(\omega))}{\varepsilon}\Big|^2 ds\Big] ,
\end{align*}
\begin{align*}
\textrm{For }\eqref{eq:P(t),4} \Rightarrow& \bE\Big[ 2\int_0^T \Big| \Big\langle P_\varepsilon(s)(\omega), \frac{b(s, \omega+\varepsilon h, X(s)(\omega+\varepsilon h)) - b(s, \omega+\varepsilon h, X(s)(\omega))}{\varepsilon} \Big\rangle\Big| ds\Big]
\\
\leq& 2L \int_0^T \bE\Big[ \|P_\varepsilon\|_{\infty,s}^2\Big] ds,
\end{align*}
\begin{align*}
\textrm{For }\eqref{eq:P(t),5} \Rightarrow& \bE\Big[ 2\int_0^T \Big| \Big\langle P_\varepsilon(s)(\omega), \frac{b(s, \omega+\varepsilon h, X(s)(\omega)) - b(s, \omega, X(s)(\omega))}{\varepsilon} \Big\rangle \Big| ds\Big]
\\
\leq& \frac{\bE[ \| P_\varepsilon\|_{\infty}^2]}{n} + n\bE\Big[ \Big( \int_0^T \Big| \frac{b(s, \omega+\varepsilon h, X(s)(\omega)) - b(s, \omega, X(s)(\omega))}{\varepsilon} \Big| ds\Big)^2 \Big], 
\end{align*}
\begin{align*}
\textrm{For }\eqref{eq:P(t),6} \Rightarrow& \bE\Big[ \sup_{t\in[0,T]} 2\int_0^t \Big\langle P_\varepsilon(s), \frac{\sigma(s, \omega+\varepsilon h, X(s)(\omega+\varepsilon h)) - \sigma(s, \omega+\varepsilon h, X(s)(\omega))}{\varepsilon} dW(s) \Big\rangle \Big]
\\
\leq& 2C_1  \bE\Big[ \|P_\varepsilon\|_{\infty} \Big(\int_0^T \Big|\frac{\sigma(s, \omega+\varepsilon h, X(s)(\omega+\varepsilon h)) - \sigma(s, \omega+\varepsilon h, X(s)(\omega))}{\varepsilon} \Big|^2 ds\Big)^{1/2} \Big]
\\
\leq& \frac{\bE[ \| P_\varepsilon\|_{\infty}^2]}{n} + nC_1^2\bE\Big[\int_0^T \Big|\frac{\sigma(s, \omega+\varepsilon h, X(s)(\omega+\varepsilon h)) - \sigma(s, \omega+\varepsilon h, X(s)(\omega))}{\varepsilon} \Big|^2 ds\Big]
\\
\leq& \frac{\bE[ \| P_\varepsilon\|_{\infty}^2]}{n} + nC_1^2\int_0^T \bE\Big[\|P_\varepsilon\|_{\infty, s}^2\Big] ds,
\end{align*}
\begin{align*}
\textrm{For }\eqref{eq:P(t),7} \Rightarrow& \bE\Big[ \sup_{t\in[0,T]} 2\int_0^t \Big\langle P_\varepsilon(s), \frac{\sigma(s, \omega+\varepsilon h, X(s)(\omega)) - \sigma(s, \omega, X(s)(\omega))}{\varepsilon} dW(s)\Big\rangle \Big]
\\
\leq& \frac{\bE[ \| P_\varepsilon\|_{\infty}^2]}{n} + nC_1\bE\Big[\int_0^T \Big|\frac{\sigma(s, \omega+\varepsilon h, X(s)(\omega)) - \sigma(s, \omega, X(s)(\omega))}{\varepsilon} \Big|^2 ds\Big],
\end{align*}
\begin{align}
\nonumber
\textrm{For }\eqref{eq:P(t),8} \Rightarrow& \bE\Big[ \int_0^T \Big| \frac{\sigma(s, \omega+\varepsilon h, X(s)(\omega+\varepsilon h)) - \sigma(s, \omega, X(s)(\omega))}{\varepsilon} \Big|^2 ds\Big]
\\
\label{eq:P(t),9}
\leq& 2\bE\Big[ \int_0^T \Big| \frac{\sigma(s, \omega+\varepsilon h, X(s)(\omega+\varepsilon h)) - \sigma(s, \omega+\varepsilon h, X(s)(\omega))}{\varepsilon} \Big|^2 ds\Big]
\\
\nonumber
&+ 2\bE\Big[ \int_0^T \Big| \frac{\sigma(s, \omega+\varepsilon h, X(s)(\omega)) - \sigma(s, \omega, X(s)(\omega))}{\varepsilon} \Big|^2 ds\Big],
\end{align}
and finally that 
$\eqref{eq:P(t),9}\leq 2L^2 \int_0^T \bE\Big[\|P_\varepsilon\|_{\infty, s}^2\Big] ds$.

Combining all these inequalities and choosing $n=6$, we have  
\begin{align*}
\frac16 {\bE\Big[ \| P_\varepsilon\|_{\infty}^2\Big]} \leq \bE\Big[ \| A_\varepsilon\|_{\infty}^2\Big] + \bar{C_1} \int_0^T \bE\Big[ \| P_\varepsilon\|_{\infty, s}^2\Big] ds,
\end{align*}
where $\bE\big[ \| A_\varepsilon\|_{\infty}^2\big] = O(1)$ as $\varepsilon \to 0$. Gr\"onwall's inequality yields that $\bE[\, \| P_\varepsilon\|_{\infty}^2 ] = O(1)$ as $\varepsilon \to 0$. 
\end{proof}

\subsubsection*{Stochastic Gateaux Differentiability of $X$}
Next we prove the convergence in probability statement of Definition \ref{definition:StochasticallyGateauxDiff}. 

\begin{theorem}
\label{theo:GateauxDifferentiabilitySDE}
Let $X$ be solution to the SDE \eqref{eq:SDE} under Assumption \ref{Assumption:2} and let $h\in H$.  Then  
 we have as $\varepsilon\to 0$ 
\begin{align*}
\Big\| \frac{X(\cdot)(\omega+\varepsilon h) - X(\cdot)(\omega)}{\varepsilon} - \int_0^\cdot M_s(\cdot)(\omega) \dot{h}(s) ds \Big\|_\infty \xrightarrow{\bP} 0.
\end{align*}
Hence $X$ satisfies  Definition \ref{definition:StochasticallyGateauxDiff}, i.e.~is Stochastically G\^ateaux differentiable.
\end{theorem}

\begin{proof} 

Let $t\in[0,T]$. To make the proof more readable we introduce several shorthand notations $M^h$, $P_\varepsilon$ and $Y_\varepsilon$, to denote increments and its differences, namely, define 
$$
M^h(t)(\omega):= \int_0^t M_s(t)(\omega) \dot{h}(s) ds, \quad 
P_\varepsilon(t)(\omega):=\tfrac{X(t)(\omega+\varepsilon h) - X(t)(\omega)}{\varepsilon},$$
and $Y_\varepsilon(t)(\omega):=P_\varepsilon(t)(\omega) - M^h(t)(\omega)$. The proof's goal is to show that $\| Y_\varepsilon(\cdot)(\omega) \|_\infty \xrightarrow{\bP} 0$ as $\varepsilon \searrow 0$.

Methodologically, we write out the SDE for $Y_\varepsilon(t)(\omega)=P_\varepsilon(t)(\omega) - M^h(t)(\omega)$ which we then break into a sequence of terms that are manipulated individually to yield an final inequality amenable to our Gr\"onwall type result for Convergence in Probability of Proposition \ref{Proposition:GronwallConProb}.

Firstly, we have
\begin{align*}
P_\varepsilon(t)(\omega) =& \int_0^t \sigma(s, \omega+\varepsilon h, X(s)(\omega+\varepsilon h)) \dot{h}(s) ds
\\
&+\int_0^t \Big[ b(s, \omega+\varepsilon h, X(s)(\omega + \varepsilon h)) - b(s, \omega, X(s)(\omega))\Big] ds
\\
&+\int_0^t \Big[ \sigma(s, \omega+\varepsilon h, X(s)(\omega + \varepsilon h)) - \sigma(s, \omega, X(s)(\omega))\Big] dW(s).
\end{align*}
This would mean we can decompose the SDE for $Y_\varepsilon=P_\varepsilon - M^h$  as 
\begin{align}
\nonumber
Y_\varepsilon(t)(\omega)&   =P_\varepsilon(t)(\omega) - M^h(t)(\omega)= \frac{X(t)(\omega+ \varepsilon h) - X(t)(\omega)}{\varepsilon } - M^h(t)(\omega)
\\
\label{eq:MallDerivConProb1}
&
= \int_0^t \Big[ \sigma(s,\omega+ \varepsilon h, X(s)(\omega+ \varepsilon h)) - \sigma(s, \omega, X(s)(\omega))\Big] \dot{h}(s) ds
\\
\label{eq:MallDerivConProb2}
&+\int_0^t \Big[ \frac{b(s, \omega+\varepsilon h, X(s)(\omega)) - b(s, \omega, X(s)(\omega)) }{\varepsilon} -  \int_0^s U(r, s, \omega)\dot{h}(r)dr \Big]ds
\\
\label{eq:MallDerivConProb3}
&+\int_0^t \Big[ \frac{\sigma(s, \omega+\varepsilon h, X(s)(\omega)) - \sigma(s, \omega, X(s)(\omega)) }{\varepsilon} -  \int_0^s V(r, s, \omega)\dot{h}(r)dr \Big]dW(s)
\\
\label{eq:MallDerivConProb4}
&+ \int_0^t \Big[\int_0^1 \nabla_xb(s, \omega+\varepsilon h, \Xi(s)) d\xi - \nabla_x b(s, \omega, X(s)(\omega)) \Big] P_\varepsilon(s)(\omega) ds
\\
\label{eq:MallDerivConProb5}
&+ \int_0^t \Big[ \int_0^1 \nabla_x\sigma(s, \omega+ \varepsilon h, \Xi(s)) d\xi - \nabla_x \sigma(s, \omega, X(s)(\omega)) \Big] P_\varepsilon(s)(\omega) dW(s)
\\
\nonumber
&+\int_0^t \nabla_xb(s, \omega, X(s)(\omega)) Y_\varepsilon(s)(\omega) ds
+\int_0^t \nabla_x\sigma(s, \omega, X(s)(\omega)) Y_\varepsilon(s)(\omega) dW(s),
\end{align}
where $\Xi(\cdot) = X(\cdot)(\omega) + \xi[X(\cdot)(\omega+\varepsilon h) - X(\cdot)(\omega)]$.

Then we take $\sup$ over $t\in[0,T]$. Notice that we will not use an It\^o type formula on the SDE, but proving convergence for each of the individual terms. 

Firstly we consider the mean convergence of \eqref{eq:MallDerivConProb1},
\begin{align*}
\bE\Big[ \Big( \int_0^T &\Big| \Big[ \sigma(s,\omega+ \varepsilon h, X(s)(\omega+ \varepsilon h)) - \sigma(s, \omega, X(s)(\omega))\Big] \dot{h}(s) \Big| ds \Big)^2 \Big]
\\
\leq& \|\dot{h}\|_2^2 \cdot \bE\Big[ \int_0^T \Big| \sigma(s,\omega+ \varepsilon h, X(s)(\omega+ \varepsilon h)) - \sigma(s, \omega, X(s)(\omega)) \Big|^2 ds  \Big]
\\
\leq& 2\|\dot{h}\|_2^2 \Big( \bE\Big[ \int_0^T \Big| \sigma(s,\omega+\varepsilon h, X(s)(\omega)) - \sigma(s, \omega, X(s)(\omega)) \Big|^2 ds\Big] 
\\
&+ 2L^2T \bE\Big[ \|X(\omega+ \varepsilon h) - X(\omega)\|_\infty^2\Big] \Big)
\\
\leq& O\Big(\varepsilon^2\Big) +  O\Big(\varepsilon^2\Big), 
\end{align*}
hence this random variable converges to zero in mean square as $\varepsilon \to 0$. 

The term \eqref{eq:MallDerivConProb2} converges in mean from Assumption \ref{Assumption:2} since as $\varepsilon \to 0$
\begin{align*}
\bE\Big[ \int_0^T \Big| \frac{b(s, \omega+\varepsilon h, X(s)(\omega)) - b(s, \omega, X(s)(\omega)) }{\varepsilon} -  \int_0^s U(r, s, \omega)\dot{h}(r)dr \Big|ds\Big] \to 0.
\end{align*}
The term \eqref{eq:MallDerivConProb3} converges in mean from Assumption \ref{Assumption:2}, namely as $\varepsilon \to 0$ 
\begin{align*}
\bE\Big[& \sup_{t\in[0,T]} \Big| \int_0^t \Big[ \frac{\sigma(s, \omega+\varepsilon h, X(s)(\omega)) - \sigma(s, \omega, X(s)(\omega)) }{\varepsilon} -  \int_0^s V(r, s, \omega)\dot{h}(r)dr \Big]dW(s) \Big| \Big]
\\
&\leq C_1 \bE\Big[ \Big( \int_0^T \Big| \frac{\sigma(s, \omega+\varepsilon h, X(s)(\omega)) - \sigma(s, \omega, X(s)(\omega)) }{\varepsilon} -  \int_0^s V(r, s, \omega)\dot{h}(r)dr \Big|^2 ds \Big)^{\tfrac{1}{2}} \Big]
\to 0.
\end{align*}

For equation \eqref{eq:MallDerivConProb4}, we are not able to use mean convergence arguments because the terms $\nabla_x b(s, \omega, x)$ have polynomial growth in $x$ and we will not necessarily have enough finite moments to ensure that this term can be dominated. We already have $\lim_{\varepsilon \to 0} \bE[\|X(\omega+\varepsilon h) - X(\omega)\|_\infty]=0$, so clearly we also have convergence in probability. Also by Proposition \ref{proposition:ContinuityCamMartin}, we have
$$
\int_0^T \Big|\nabla_x b(s, \omega+\varepsilon h, X(s)(\omega+\varepsilon h)) - \nabla_x b(s, \omega, X(s)(\omega))\Big|ds \xrightarrow{\bP} 0. 
$$
for any choice of $x\in \bR^d$. Therefore, by continuity of $\nabla_xb$ from Assumption \ref{Assumption:2}, we get 
\begin{align*}
\int_0^T\Big|\int_0^1 \nabla_xb(s, \omega+\varepsilon h,\Xi(s)) d\xi - \nabla_x b(s, \omega, X(s)(\omega))\Big|ds \xrightarrow{\bP} 0,
\quad \textrm{as }\varepsilon \to0. 
\end{align*}
Since we also have finite moments of  ${\|X(\omega+\varepsilon h) - X(\omega)\|_{\infty}}/{\varepsilon}$ by Proposition \ref{proposition:MallCalcRayAbsoluteCont}, we can conclude that \eqref{eq:MallDerivConProb4} converges to zero in probability. 


For \eqref{eq:MallDerivConProb5} we know that $\sigma$ is Lipschitz so we have $\nabla_x \sigma$ is bounded. Hence, we won't have the same integrability issues as with \eqref{eq:MallDerivConProb4}. Therefore, we use convergence in mean. By the Burkholder-Davis-Gundy Inequality and recalling Proposition \ref{proposition:MallCalcRayAbsoluteCont} we get
\begin{align*}
\bE\Big[ &\sup_{t'\in[0,T]} \Big| \int_0^{t'} \Big( \int_0^1 \nabla_x\sigma(s, \omega+ \varepsilon h, \Xi(s)) d\xi - \nabla_x \sigma(s, \omega, X(s)(\omega)) \Big) \cdot P_\varepsilon(s)(\omega) dW(s)\Big| \Big]
\\
\leq& C_1 \bE\Big[ \Big( \int_0^T \Big| \int_0^1 \nabla_x\sigma(s, \omega+ \varepsilon h, \Xi(s)) d\xi - \nabla_x \sigma(s, \omega, X(s)(\omega)) \Big|^2 \cdot \Big| P_\varepsilon(s)(\omega) \Big|^2 ds \Big)^{\tfrac{1}{2}} \Big] 
\\
\leq& C_1 \bE\Big[ \Big( \int_0^T \Big| \int_0^1 \nabla_x\sigma(s, \omega+ \varepsilon h, \Xi(s)) d\xi - \nabla_x \sigma(s, \omega, X(s)(\omega)) \Big|^2ds \Big)^{\tfrac{1}{2}} \cdot  \| P_\varepsilon(\omega)\|_\infty  \Big] 
\\
\leq& C_1 \bE\Big[ \| P_\varepsilon(\omega) \|_{\infty}^2\Big]^{\tfrac{1}{2}} \cdot \bE\Big[ \int_0^T \Big| \int_0^1 \nabla_x\sigma(s, \omega+ \varepsilon h, \Xi(s)) d\xi - \nabla_x \sigma(s, \omega, X(s)(\omega)) \Big|^2ds  \Big]^{\tfrac{1}{2}}. 
\end{align*}
In the same way as earlier, by continuity of $\nabla_x\sigma $ from Assumption \ref{Assumption:2} and Proposition \ref{proposition:ContinuityCamMartin} we get
\begin{align*}
\int_0^T \Big| \int_0^1 \nabla_x\sigma(s, \omega+ \varepsilon h, \Xi(s)) d\xi - \nabla_x \sigma(s, \omega, X(s)(\omega)) \Big|^2 ds \xrightarrow{\bP} 0.
\end{align*}

Also, by boundedness of $\nabla_x \sigma$, we have the immediate domination
$$
\int_0^T \Big| \int_0^1 \nabla_x\sigma(s, \omega+ \varepsilon h, \Xi(s)) d\xi - \nabla_x \sigma(s, \omega, X(s)(\omega)) \Big|^2 ds\leq 4L^2 T,
$$
so we clearly have uniform integrability of all orders. Hence
\begin{align*}
\lim_{\varepsilon \to 0} \bE\Big[ \int_0^T \Big| \int_0^1 \nabla_x\sigma(s, \omega+ \varepsilon h, \Xi(s)) d\xi - \nabla_x \sigma(s, \omega, X(s)(\omega)) \Big|^2ds  \Big]^{\tfrac{1}{2}} = 0. 
\end{align*}

Finally, the SDE for the process $Y_\varepsilon(t)(\omega)$ can be written in the convenient form
\begin{align*}
Y_\varepsilon(t)(\omega) = A_\varepsilon(\omega) + \int_0^t \nabla_x b(s, \omega, X(s)(\omega)) Y_\varepsilon(s)(\omega) ds + \int_0^t \nabla_x \sigma(s, \omega, X(s)(\omega)) Y_\varepsilon(s)(\omega) dW(s),
\end{align*}
where the sequence $A_\varepsilon$ is a sequence of random variables which converge to zero in probability. By Proposition \ref{Proposition:GronwallConProb} the random variable $\|Y_\varepsilon\|_\infty$ converges in probability to zero as $\varepsilon \to 0$. 
\end{proof}

\subsubsection*{Strong Stochastic G\^ateaux Differentiability}

\begin{theorem}
\label{theo:SGateauxDifferentiabilitySDE}
Let $X$ be solution to the SDE \eqref{eq:SDE} under Assumption \ref{Assumption:2}. Then for any $h\in H$
\begin{align*}
\lim_{\varepsilon \to 0} \bE\Bigg[ \Big\| \frac{X(\omega+\varepsilon h) - X(\omega)}{\varepsilon} - M^h(\omega)\Big\|_\infty \Bigg] =0. 
\end{align*}
Hence $X$ satisfies Equation \eqref{eq:StrongStochasticGatDif}, i.e.~is Strong Stochastically G\^ateaux differentiable.
\end{theorem}

\begin{proof}
By Theorem \ref{theo:GateauxDifferentiabilitySDE}, we have convergence in Probability. Combining this with Proposition \ref{proposition:MallCalcRayAbsoluteCont} and Theorem \ref{theo:ExistenceUniqueness for Candidate DX}, we have
$$
\bE\Big[\, \Big\| \frac{X(\omega+\varepsilon h) - X(\omega)}{\varepsilon}\Big\|_\infty^2 \Big], 
\quad 
\bE\Big[\, \| M^h(\omega)\|_\infty^2\Big] < \infty. 
$$
Apply Corollary \ref{corollary:HowToRAC} to conclude. 
\end{proof}

\begin{remark}
Although convergence in probability may seem to be rather a weak result relative to the much stronger Almost sure convergence or convergence in mean square, it is actually the case that we now have both. After all, we proved that the sequence of random variables $\big({X(\cdot)(\omega+\varepsilon h) - X(\cdot)(\omega)}\big)/{\varepsilon}$ have uniform finite $p$ moments over $\varepsilon$ and the limit $D^h X(\cdot)$ has finite $p$ moments. Therefore, by standard probability theory we have mean square convergence. 

Moreover, convergence in probability implies existence of a subsequence which converges almost surely. This, combined with the Ray Absolute continuity ensures uniqueness of the limit for \emph{any} choice of subsequence which implies almost sure convergence. 
\end{remark}

\subsubsection*{Proof of the Malliavin differentiability result, Theorem \ref{theo-Mall-diff-monotone-SDEs}}

\begin{proof}[Proof of Theorem \ref{theo-Mall-diff-monotone-SDEs}]

The proof is straightforward and follows from Theorem \ref{theo:SGateauxDifferentiabilitySDE} and Theorem \ref{theo:Reveillacs results}.   
Further, the Malliavin Derivative satisfies the SDE \eqref{eq:SDEMallDeriv} which has a unique solution as proved in Theorem \ref{theo:ExistenceUniqueness for Candidate DX}. 
\end{proof}

\subsection[Proofs of the 2nd main result]{Proofs of the 2nd main result - Theorem \ref{theo:WhatsHere}}
\label{sec:An Example of such a differentiable SDE}

In order to prove the Malliavin differentiability (Theorem \ref{theo-Mall-diff-monotone-SDEs}) under the weakest possible conditions, we only assumed enough properties to ensure convergence of the Stochastic G\^ateaux Derivatives. However, the Stochastic G\^ateaux differentiability conditions for $b$ and $\sigma$ do not require that $b$ and $\sigma$ are Malliavin differentiable. These conditions need to be checked by the user on a case-by-case basis. Under slightly stronger conditions, but much easier to verify, we present an argument to establish integrability and convergence of $b$ and $\sigma$ to prove Theorem \ref{theo-Mall-diff-monotone-SDEs}. 

In \cite{geiss2016malliavin}, there is a discussion about how much continuity is required for the spacial variable in the Malliavin Derivatives of $b$ and $\sigma$ in order to prove Malliavin Differentiability of the solution $X$. The authors prove results similar  to those in this paper using much weaker continuity condition, but in doing so assume the integrability of the terms $D_sb(t, \omega, X(t))$ and $D_s\sigma(t, \omega, X(t))$. In our manuscript, we were unable to ensure integrability of $b$ and $\sigma$ evaluated at $X$ without the Lipschitz (or otherwise tractable assumptions). Weaker continuity conditions would have allowed for examples where $b(t, \omega, X(t)(\omega)$ and $\sigma(t, \omega, X(t)(\omega))$ were not adequately integrable. Therefore, for easy to check conditions, we work under Assumption \ref{Assumption:3} (iii') and (iv') (see Remark \ref{rem:cautionarytale}). 

For simplicity, we introduce Assumption \ref{Assumption:f} which contains all of the relevant properties of Assumption \ref{Assumption:3} that we require for this section. The function $f$ represents $b$ or $\sigma$ depending on the choice of $m$.

\begin{assumption} 
\label{Assumption:f}
Let $m\in\{ 1, 2\}$. Suppose that $f:[0,T]\times \Omega\times \bR^d \to \bR^d$ such that 
\begin{enumerate}[(i)]
\item $\forall x\in\bR^d$ $f(\cdot, \cdot, x) \in \bD^{1,p}(L^m([0,T]; \bR^d))$. 
\item $f$ is Locally Lipschitz in the spacial variable i.e~$\exists L_N>0$ such that $\forall x, y\in\bR^d$ such that $|x|, |y|\leq N$ and $\forall t\in [0,T]$, 
\begin{align*}
|f(t, \omega, x) - f(t, \omega, y)|\leq L_N |x-y| \quad \bP\mbox{-almost surely. }
\end{align*}
\item $Df$ are Lipschitz in their spatial variables i.e.~$\exists L>0$ constant such that $\forall (s, t) \in  [0,T]^2$ and $\forall x, y\in \bR^d$, 
\begin{align*}
|D_s f(t, \omega, x) - D_sf(t, \omega, y)| \leq L |x-y| \quad \bP\mbox{-almost surely.}
\end{align*}
\end{enumerate}
\end{assumption}

\subsubsection*{Integrability and indistinguishability of the Malliavin Derivative}

\begin{lemma}
\label{Lemma:FinMomMalDeriv:b+sigma}
Let $m\in\{ 1, 2\}$ and $p>2$. Let $X$ be solution to the SDE \eqref{eq:SDE} under Assumption \ref{Assumption:1} and let $f$ satisfy Assumption \ref{Assumption:f}.   
Then
\begin{align*}
\bE\Big[ \Big( \int_0^T \Big( \int_0^t | D_s f(t, \omega, X(t)(\omega)) |^2 ds\Big)^{\tfrac{m}{2}} dt \Big)^{\tfrac{p}{m}} \Big]<\infty.
\end{align*}
\end{lemma}

\begin{proof}
By the definition of $\bD^{1, p}(L^m([0,T]; \bR^d))$ we have for any $t\in[0,T]$
\begin{align*}
\bE\Big[ \Big( \int_0^T \Big( \int_0^t | D_s f(t, \omega, 0) |^2 ds\Big)^{\tfrac{m}{2}} dt \Big)^{\tfrac{p}{m}} \Big] <\infty. 
\end{align*} 
Therefore for some constant $C$ (depending on $p$, $m$, $T$, $L$) we have 
\begin{align*}
\bE\Big[\Big( \int_0^T \Big( \int_0^t &\Big| D_sf(t, \omega, X(t)(\omega) ) \Big|^2 ds\Big)^{\tfrac{m}{2}} dt \Big)^{\tfrac{p}{m}} \Big] 
\\
\leq& 
2^{\tfrac{p-m}{m}} C \Bigg(
\bE\Big[ \Big( \int_0^T \Big( \int_0^t |D_s f(t, \omega, 0)|^2 ds \Big)^{\tfrac{m}{2}} dt \Big)^{\tfrac{p}{m}} \Big]
+ \bE\Big[ \| X\|_\infty^p\Big]
\Bigg)<\infty. 
\end{align*}
\end{proof}

We have by Assumption \ref{Assumption:f} that for every $x\in \bR^d$ the random field $f(\cdot, \cdot, x)$ is a Malliavin differentiable process. However, it is not immediate that we have the same for $f(\cdot, \cdot, X(\cdot)(\cdot))$. We first prove an indistinguishability property for when we replace $x$ by $X(\cdot)(\omega)$.  

\begin{lemma}
\label{lemma:MallCalcIndistinguishability1}
Let $m\in\{ 1, 2\}$ and $p>2$. Let $X$ be solution to the SDE \eqref{eq:SDE} under Assumption \ref{Assumption:1}. Let $f$ satisfy Assumption \ref{Assumption:f} and recall the directional derivative notation introduced previously, $D^hF = \langle DF, h\rangle$ for any choice of $h\in H$. 

Then, for $h\in H$ we have, $(t, \omega)$-almost surely that
\begin{align*}
f\Big(t, \omega+\varepsilon h, X(t)(\omega)\Big) - f\Big(t, \omega, X(t)(\omega)\Big) 
= \int_0^\varepsilon D^h f(t, \omega+rh, X(t)(\omega)) dr.
\end{align*}
\end{lemma}

\begin{proof}
We have that $\forall x\in\bR^d$ that $\exists C_x \subset [0,T]\times \Omega$ with $\bE[\int_0^T \1_{C_x}(t, \omega) dt]=0$, dependent on the choice of $x$, for which $\forall (t, \omega) \in [0,T]\times \Omega \backslash C_x$ that

\begin{equation}
\label{eq:Indistinguish1}
f(t, \omega+\varepsilon h, x) - f(t, \omega, x) = \int_0^\varepsilon D^h f(t, \omega+rh, x) dr.
\end{equation}
We wish to prove that we can choose a null set $C$ which is independent of $x$ outside of which the equality holds. To do this, it suffices to prove almost sure continuity with respect to $x$ of both the left and right hand side of \eqref{eq:Indistinguish1}. 

Almost sure continuity of the left hand side is immediate since $f$ is locally Lipschitz. For the right hand side, we use the Lipschitz properties of the Malliavin derivative.  Let $r_i$ be an enumeration of the rationals $\bQ^d$. Then we have $\bigcup_{i} C_{r_i}$ is also a null set since it is the countable union of null sets. Then for $(t, \omega) \in [0,T]\times \Omega \backslash \Big( \bigcup_{i} C_{r_i}\Big)$ and $\forall x\in \bQ^d$ equation \eqref{eq:Indistinguish1} holds. Then by the continuity of $f$ and its Malliavin derivative we conclude that this also holds $\forall x\in \bR^d$. 
\end{proof}

\subsubsection*{Strong Stochastic G\^ateaux Differentiability} 
\begin{lemma}
\label{lemma:MallFinSecMom1}
Let $m\in\{ 1, 2\}$ and $p>2$. Let $X$ be solution to the SDE \eqref{eq:SDE} under Assumption \ref{Assumption:1}. Let $f$ satisfy Assumption \ref{Assumption:f}. Then 
\begin{align*}
\bE\Bigg[  \Big(\int_0^T \Big| \frac{f(t, \omega+\varepsilon h, X(t)(\omega)) - f(t, \omega, X(t)(\omega))}{\varepsilon} \Big|^m dt\Big)^{\tfrac{2}{m}}  \Bigg] = O(1),\quad \textrm{as } \varepsilon\searrow 0. 
\end{align*}
\end{lemma}

\begin{proof}
Fix $\varepsilon >0$. By Lemma \ref{lemma:MallCalcIndistinguishability1}, for almost all $\omega\in \Omega$ we have that
\begin{align*}
\int_0^T \big| f(t, \omega+\varepsilon h, X(t)(\omega)) - f(t, \omega, X(t)(\omega))\big|^m dt = \int_0^T \Big|\int_0^\varepsilon D^h f (t, \omega+rh, X(t)(\omega)) dr\Big|^m dt. 
\end{align*}

Arguing from this, we have with the help of the directional derivative $D^h$, Jensen and reverse Jensen inequality,
\begin{align*}
\nonumber
\Big( \int_0^T & \big| f(t, \omega+\varepsilon h, X(t)(\omega)) - f(t, \omega, X(t)(\omega))\big|^m dt\Big)^{\tfrac{2}{m}} 
\\ \nonumber
&
=  \Big(\int_0^T \Big| \int_0^\varepsilon D^h f(t, \omega+rh, X(t)(\omega)) dr\Big|^m dt\Big)^{\tfrac{2}{m}} 
\\
\nonumber
&\leq \varepsilon \int_0^\varepsilon \Big( \int_0^T |D^h f(t, \omega+rh, X(t)(\omega))|^m dt\Big)^{\tfrac{2}{m}} dr 
\\
\nonumber
&\leq \varepsilon \|\dot{h}\|_2^2 \int_0^\varepsilon  \Big( \int_0^T \Big( \int_0^t | D_s f(t, \omega+rh, X(t)(\omega)) |^2ds \Big)^{\tfrac{m}{2}} dt \Big)^{\tfrac{2}{m}} dr 
\\ 
&\leq 2^{\tfrac{2}{m}} \varepsilon \|\dot{h}\|_2^2 \Bigg(\int_0^\varepsilon  \Big( \int_0^T \Big( \int_0^t | D_s f(t, \omega+rh, 0) |^2ds \Big)^{\tfrac{m}{2}} dt  \Big)^{\tfrac{2}{m}} dr 
\varepsilon
\|X(\omega)\|_\infty^{2} \cdot T^{\tfrac{2}{m}+1}\Bigg). 
\end{align*}
Therefore
\begin{align}
\nonumber
\bE\Big[ \frac{1}{\varepsilon^2} &\big(\int_0^T \big| f(t, \omega+\varepsilon h, X(t)(\omega)) - f(t, \omega, X(t)(\omega)) \big|^m dt\big)^{\tfrac{2}{m}}  \Big] 
\\
\label{eq:MallFinSecMom1.1}
&\leq  2^{\tfrac{2}{m}}  \|\dot{h}\|_2^2 \bE\Big[  \frac{1}{\varepsilon} \int_0^\varepsilon  \Big( \int_0^T \Big( \int_0^t | D_s f(t, \omega+rh, 0) |^2ds \Big)^{\tfrac{m}{2}} dt  \Big)^{\tfrac{2}{m}} dr \Big]
\\
\nonumber
&\hspace{20pt}+ 2^{\tfrac{2}{m}} \|\dot{h}\|_2^2  T^{\tfrac{2}{m}+1} \bE\Big[ \|X(\omega)\|_\infty^{2}\Big] . 
\end{align}
We estimate term \eqref{eq:MallFinSecMom1.1} as follows and with the help of Proposition \ref{proposition:CamMarForm1}
\begin{align*}
\eqref{eq:MallFinSecMom1.1} \leq& 2^{\tfrac{2}{m}}  \|\dot{h}\|_2^2 \frac{1}{\varepsilon} \int_0^\varepsilon \bE\Big[  \int_0^T \Big( \int_0^t | D_s f(t, \omega, 0) |^2ds \Big)^{\tfrac{m}{2}}  \cdot \cE(r \dot h)(t) dt  \Big]^{\tfrac{2}{m}}  dr
\\
\leq& 2^{\tfrac{2}{m}}  \|\dot{h}\|_2^2 \bE\Big[ \Big( \int_0^T \Big( \int_0^t | D_s f(t, \omega, 0) |^2ds \Big)^{\tfrac{m}{2}}dt\Big)^{\tfrac{p}{m}} \Big]^{\tfrac{2}{p}}  \frac{1}{\varepsilon} \int_0^\varepsilon \bE\Big[ \| \cE(r\dot h)(\cdot) \|_\infty^{\tfrac{p}{p-m}}\Big]^{\tfrac{2(p-m)}{pm}} dr
\\
<& O(1),
\end{align*}
with $\cE(r \dot h)$ denoting the stochastic exponential of $r \dot h$ as introduced in \eqref{eq:DDexponential}.
\end{proof}
 
\begin{lemma}
\label{lemma:MallDerivprelimConverge}
Let $m\in\{ 1, 2\}$ and $p>2$. Let $X$ be solution to the SDE \eqref{eq:SDE} under Assumption \ref{Assumption:1}. Let $f$ satisfy Assumption \ref{Assumption:f}. Then for $h\in H$ and any $\delta>0$
\begin{equation}
\label{eq:MallDerivprelimConverge}
\lim_{\varepsilon \to 0} \bP\Big[ \Big( \int_0^T \Big| \frac{1}{\varepsilon}\int_0^\varepsilon D^h f\big(t, \omega+rh, X(t)(\omega)\big) dr - D^h f\big(t, \omega, X(t)(\omega)\big) \Big|^m dt >\delta \Big] = 0.
\end{equation}
\end{lemma}

\begin{proof}
By Proposition \ref{proposition:ContinuityCamMartin}, we know that for any $\delta>0$ that
\begin{equation}
\label{eq:MallDerivprelimConverge1}
\lim_{\varepsilon \to 0} \bP\Big[ \int_0^T \Big| D^hf(t, \omega + \varepsilon h, X(t)(\omega+\varepsilon h)) - D^hf(t, \omega, X(t)(\omega))\Big|^m dt > \delta \Big] = 0. 
\end{equation}
Similarly
$$
\lim_{\varepsilon \to 0} \bP\Big[ \| X(\omega+\varepsilon h) - X(\omega)\|_\infty > \delta \Big] = 0,
$$
so by Lipschitz continuity of $Df$ we also have
\begin{equation}
\label{eq:MallDerivprelimConverge2}
\lim_{\varepsilon \to 0} \bP\Big[ \int_0^T \Big| D^hf(t, \omega + \varepsilon h, X(t)(\omega+\varepsilon h)) - D^hf(t, \omega+\varepsilon h, X(t)(\omega))\Big|^m dt > \delta \Big] = 0. 
\end{equation}
Combining Equations \eqref{eq:MallDerivprelimConverge1} and \eqref{eq:MallDerivprelimConverge2}, we conclude
$$
\lim_{\varepsilon \to 0} \bP\Big[ \int_0^T \Big| D^hf(t, \omega + \varepsilon h, X(t)(\omega)) - D^hf(t, \omega, X(t)(\omega))\Big|^m dt > \delta \Big] = 0.
$$
Next, using the Fundamental Theorem of Calculus, we also have
$$
\lim_{\varepsilon \to 0} \bP\Big[ \int_0^T \Big| \frac{1}{\varepsilon} \int_0^\varepsilon D^hf\big(t, \omega + r h, X(t)(\omega)\big)dr - D^hf\big(t, \omega, X(t)(\omega)\big)\Big|^m dt > \delta \Big] = 0.
$$

\end{proof}

The next result establishes the Strong Stochastic G\^ateaux differentiability, see Definition \ref{def:Reveillacs results}. 
\begin{lemma}
\label{lemma:MallDerivConverge}
Let $m\in\{ 1, 2\}$ and $p>2$. Let $X$ be solution to the SDE \eqref{eq:SDE} under Assumption \ref{Assumption:1}. Let $f$ satisfy Assumption \ref{Assumption:f}. Then for $h\in H$
\begin{align*}
\lim_{\varepsilon\to0} \bE\Big[ \Big( \int_0^T \Big| \frac{f(t, \omega+\varepsilon h, X(t)(\omega)) - f(t, \omega, X(t)(\omega))}{\varepsilon} - D^h f(t, \omega, X(t)(\omega))\Big|^m dt\Big)^{\tfrac{1}{m}} \Big] = 0.
\end{align*}
\end{lemma}

\begin{proof}
First, using Lemma \ref{lemma:MallCalcIndistinguishability1}, we have $\bP$-almost surely that
\begin{align*}
\int_0^T \Big|& \frac{f(t, \omega+\varepsilon h, X(t)(\omega)) - f(t, \omega, X(t)(\omega))}{\varepsilon} - D^h f(t, \omega, X(t)(\omega))\Big|^m dt
\\
&\qquad =\int_0^T \Big| \frac{1}{\varepsilon} \int_0^\varepsilon D^hf\big(t, \omega + r h, X(t)(\omega)\big)dr - D^h f(t, \omega, X(t)(\omega))\Big|^m dt.
\end{align*}
By Lemma \ref{lemma:MallDerivprelimConverge}, both sides converge to $0$ in probability (as $\varepsilon \to 0$). 

Next, by Lemma \ref{Lemma:FinMomMalDeriv:b+sigma} and Lemma \ref{lemma:MallFinSecMom1}, we have uniform $L^1$ integrability of this collection of random variables since they are bounded in $L^2$. Convergence in probability and Uniform Integrability imply convergence in mean. 
\end{proof}
 
\subsubsection*{Proof of Theorem \ref{theo:WhatsHere}}

\begin{proof}[Proof of Theorem \ref{theo:WhatsHere}]
The difference between Assumptions \ref{Assumption:2} and Assumptions \ref{Assumption:3} is (iii') and (iv'). Here we verify that $b$ and $\sigma$ satisfying Assumption \ref{Assumption:3} implies Assumptions \ref{Assumption:2}.  

Lemma \ref{Lemma:FinMomMalDeriv:b+sigma} implies Assumptions \ref{Assumption:2} (iii) is satisfied. Lemma \ref{lemma:MallDerivConverge} implies Assumptions \ref{Assumption:2} (iv) is satisfied. In this case, the identification $U,V$ with $Db$ and $D\sigma$ respectively is straightforward.This also means that the Existence proof in Theorem \ref{theo:ExistenceUniqueness for Candidate DX} holds so a solution to the SDE \eqref{eq:SDEMallDeriv2} must exist. 
\end{proof}

\section{Parametric differentiability }
\label{sec:Parametric differentiability}

In this section, we study the differentiability properties of solutions of SDEs with respect to the initial condition. For a detailed exploration of the subject of Stochastic flows, see \cite{Kunita1990}. The main contribution of this section is to prove similar results for SDEs with only locally Lipschitz and monotone coefficients as opposed to previous results which rely on a Lipschitz condition. Similar problems have been studied in \cite{RIEDEL2017283}, \cite{Cerrai2001}*{Chapter 1} and \cite{Zhang2016}.

\subsection{G\^ateaux and Frech\'et Differentiability of monotone SDEs}
We start by recalling the concept of G\^ateaux and Frech\'et Differentiability for abstract Banach Spaces.
\begin{definition}[G\^ateaux and Frech\'et Differentiability]
\label{definition:Frechet}
Let $V$ and $W$ be Banach spaces and let $U$ be an open subset of $V$. Let $f:U \to W$.  The map $f$ is \emph{G\^ateaux differentiable} at $x\in U$ in direction $h\in V$ if the limit
\begin{align*}
\lim_{\varepsilon \to 0} \frac{f(x+\varepsilon h) - f(x)}{\varepsilon} = \frac{d}{d\varepsilon} f(x + \varepsilon h),
\end{align*}
exists. The limit is called the G\^ateaux derivative in direction $h$.

The map $f$ is said to be Frech\'et differentiable at $x\in U$ if there exists a bounded linear operator $A:U\to W$ such that 
\begin{align*}
\lim_{\|h\|_V \to 0} \frac{\| f(x+h) - f(x) - Ah \|_W}{\|h\|_V} = 0.
\end{align*}
The linear operator $A$ is called the Frech\'et derivative of $f$ at $x$
\end{definition}

Let $X_\theta$ be the solution of SDE \eqref{eq:SDE}. We next show that the map $ \theta \in L^p(\cF_0; \bR^d; \bP) \mapsto X_\theta(\cdot) \in \cS^p([0,T])$ is Frech\'et differentiable.  As we will be differentiating with respect to $\theta$ for this section, we emphasize the dependency on $\theta$. 

\begin{assumption}
\label{Assumption:Frechet}
Let $b: [0,T]\times \Omega\times \bR^d \to \bR^d$ and $\sigma: [0,T] \times \Omega\times\bR^d \to \bR^{d\times m}$ satisfy Assumption \ref{Assumption:1} for some $p\geq2$. Further, suppose
\begin{enumerate}[(i)]
\item For almost all $(t, \omega)\in [0,T]\times \Omega$ we have the functions $\sigma(t, \omega, \cdot)$ and $b(t, \omega, \cdot)$ have partial derivatives in all directions. 
\item For all  $x\in \bR^d$, we have that the maps 
\begin{align*}
x \mapsto \int_0^T \Big| \nabla_x \sigma(t, \omega, x) \Big|^2 dt
\quad\textrm{and}\quad 
x \mapsto \int_0^T \Big| \nabla_x b(t, \omega, x) \Big|^2 dt
\quad \textrm{are $\bP$-almost surely continuous.}
\end{align*}
 
\end{enumerate}
\end{assumption}

\begin{theorem}
\label{theorem:Frechet}
Let $p\geq 2$ and let $1\leq q<p$.  
Let $X_\theta$ be the solution of SDE \eqref{eq:SDE} under Assumption \ref{Assumption:Frechet} in $\cS^q$. Then the map $\theta \to X_\theta$ is G\^ateaux Differentiable in direction $h$ and the derivative is equal to $F[h]$ the solution of the SDE \eqref{eq:SDEFrechet}

Further, the operator $F:L^p(\cF_0; \bR^d; \bP) \to S^q([0,T])$ is the Frech\'et derivative. 
\end{theorem}

\begin{remark}
It is important to note that we were unable to prove G\^ateaux Differentiability in the Banach space $\cS^p$. Convergence in $\cS^p$ would be equivalent to uniform integrability of the random variable 
$$
\Big\| \frac{ X_{\theta+h} - X_\theta - F[h] }{\|h\|_{L^p(\cF_0; \bR^d; \bP)}} \Big\|_\infty^p,
$$
over all possible choices of $h\in L^p(\cF_0; \bR^d; \bP)$. Unlike in the case where the coefficients are Lipschitz, see \cite{crisan2017smoothing}, this is not true. 
\end{remark}

The proof is given after several intermediary results. The first results relates to G\^ateaux differentiability and its properties, we address the Frech\'et differentiability afterwards.  For the proof once one has established G\^ateaux differentiability, extending to Frech\'et differentiability is remarkably easy. G\^ateaux differentiability is the weaker condition and is usually considered the easier property to prove.

\subsubsection*{Existence and Uniqueness for the candidate process}

\begin{theorem}
\label{thm:FrechetExistence}
Let $p\geq2$ and suppose Assumption \ref{Assumption:Frechet} holds. Let $X_\theta$ be the solution to \eqref{eq:SDE}. Let $h \in L^p( \cF_0; \bR^d; \bP)$. Then the SDE
\begin{equation}
\label{eq:SDEFrechet}
F(t)[h] = h + \int_0^t \nabla_x b\Big(s, \omega, X_\theta(s)(\omega)\Big) F(s)[h]ds + \int_0^t \nabla_x \sigma\Big(s, \omega, X_\theta(s)(\omega)\Big) F(s)[h] dW(s),
\end{equation}
has a unique solution in $\cS^p([0,T]; \bR^d)$. 
\end{theorem}

\begin{proof}
This just follows from Theorem \ref{theorem:ExistenceUniquenessMaoLinearSDE}. We simply verify that Assumption \ref{Assumption:linear} holds:
\begin{enumerate}
\item $|\nabla_x \sigma|<L$ by the Lipschitz property. Therefore, clearly $\bE\Big[ \int_0^T |\nabla_x \sigma(s, \omega, X_\theta(s))|^2 ds \Big]<\infty$. 
\item From the differentiability and the monotonicity property of $b$, we have that $\nabla_x b$ is $\bP$-almost surely negative semidefinite\footnote{We do not prove this fact; it is straightforward using inner products and the definition of derivative.}. Therefore, for $z\in\bR^d$
\begin{align*}
z^T \Big( \int_0^T \nabla_x b(s, \omega, X_\theta(s)) ds \Big) z \leq \int_0^T L |z|^2 ds \leq LT |z|^2,
\end{align*}
\end{enumerate}
Hence, using the moment estimates we conclude that 
$\bE\big[ \|F[h]\|_{\infty}^p\big] \lesssim \|h\|_{L^p(\cF_0; \bR^d; \bP)}^p.$
\end{proof}

Unlike with the Malliavin Derivative, the SDE \eqref{eq:SDEFrechet} is not a general linear stochastic differential equation. As $b$ and $\sigma$ do not have dependency on $\theta$, we do not have extra terms akin to the Malliavin derivatives $Db$ and $D\sigma$. This means that, unlike the Malliavin Derivative, $F$ has finite moments of all orders provided the initial condition has adequate integrability. 

\begin{proposition}
\label{proposition:FrechetLinear}
Let $p\geq 2$. Suppose Assumption \ref{Assumption:Frechet}. Let $X_\theta$ be the solution to \eqref{eq:SDE}. The operator $F: L^p(\cF_0; \bR^d; \bP) \to \cS^p([0,T])$ defined by $h\mapsto F[h]$ the solution of Equation \eqref{eq:SDEFrechet}, is a bounded linear operator
$
\| F[h] \|_{\cS^p} \lesssim \| h \|_{L^p(\cF_0; \bR^d; \bP)}.
$ 
\end{proposition}

\begin{proof}
Firstly, we show that $F[0](\cdot)=0_d$ a.s. ($0_d$ is the $\bR^d$-vector of zeros). Since $F[0]$ is the solution to the SDE
\begin{align*}
F(t)[0] = \int_0^t \nabla_x b(s, \omega, X_\theta(s)(\omega)) F(s)[0]ds + \int_0^t \nabla_x \sigma(s, \omega, X_\theta(s)(\omega)) F(s)[0] dW(s), 
\quad
F(0)[0] = 0
\end{align*}
and this SDE has a unique solution, we only need to show that $F[0](\cdot)=0_d$ is a solution. Clearly we have $\bP$-almost surely that
\begin{align*}
\int_0^t \nabla_x b(s, \omega, X_\theta(s)(\omega)) \cdot 0_d ds = 0 \quad \mbox{and} \quad \int_0^t \nabla_x \sigma(s, \omega, X_\theta(s)(\omega)) \cdot 0_d dW(s) = 0, 
\end{align*}
so this is immediate. 

Let $\lambda \in \bR$. Next we have
\begin{align*}
&F[h_1](t) + \lambda F[h_2](t)
\\
& = h_1 + \lambda h_2 + \int_0^t \nabla_x b(s, \omega, X_\theta(s)(\omega)) F[h_1](s)ds + \lambda \int_0^t \nabla_x b(s, \omega, X_\theta(s)(\omega)) F[h_2](s)ds 
\\
&
\qquad 
+ \int_0^t \nabla_x \sigma(s, \omega, X_\theta(s)(\omega)) F[h_1](s) dW(s) + \lambda \int_0^t \nabla_x \sigma(s, \omega, X_\theta(s)(\omega)) F[h_2](s) dW(s), 
\\
&\Big( F[h_1] + \lambda F[h_2]\Big)(t)
\\
& = (h_1 + \lambda h_2) + \int_0^t \nabla_x b(s, \omega, X_\theta(s)(\omega)) \Big( F[h_1](s) + \lambda F[h_2](s)\Big)(s) ds 
\\
&\qquad + \int_0^t \nabla_x \sigma(s, \omega, X_\theta(s)(\omega)) \Big( F[h_1](s) + \lambda F[h_1](s)\Big)(s) dW(s) , 
\end{align*}
which is the same as the SDE for $F[h_1+\lambda h_2]$. Hence, by existence and uniqueness, the two must be equal up to a null set. 

To prove boundedness we note that we have $\| F[h] \|_{\cS^p}  \lesssim \| h \|_{L^p(\cF_0; \bR^d; \bP)}$ from Theorem \ref{thm:FrechetExistence}.

\end{proof}

\subsubsection*{Differentiability of $\theta\mapsto X_{\theta}$}

It is immediate to prove the stochastic stability result that
$
\bE\big[ \| X_{\theta+h} - X_\theta\|_\infty^p \big]^{{1}/{p}} = O\big( \|h\|_{L^p} \big)
$ 
as $\|h\|_{L^p} \to 0$, see Theorem \ref{theorem:ExistenceUniquenessMao}. Hence we have
\begin{align*}
\lim_{\|h\|_{L^p} \to 0} \bE\big[ \| X_{\theta+h} - X_\theta\|_\infty^p\big] \to 0.
\end{align*}

\begin{theorem}
\label{Theorem:FrechetDerivative}
Let $p\geq2$ and $1\leq q<p$. Let $h\in L^p(\cF_0; \bR^d; \bP)$. Suppose we have Assumption \ref{Assumption:Frechet}, let $X_\theta$ be the solution of the SDE \eqref{eq:SDE} and let $F(t)[h]$ be the solution to the SDE \eqref{eq:SDEFrechet}. Then we have
\begin{align*}
\| X_{\theta+h} - X_\theta - F[h]\|_{\cS^q} = o\big(\|h\|_{L^p}\big),
\end{align*}
and therefore $F[h]$ is the G\^ateaux derivative of $X$. 
\end{theorem}

\begin{proof}
Let $t\in[0,T]$. Define $\Xi(\cdot) = X_\theta(\cdot) + \xi[X_{\theta+h}(\cdot) - X_\theta(\cdot)]$ and consider
\begin{align}
\nonumber
\frac{ X_{\theta+h}(t) - X_\theta(t) - F[h](t)}{\|h\|_{L^p}} &=  \frac{(\theta+h) - \theta - h}{\|h\|_{L^p}}
\\
\label{eq:FrechetconvProp1}
+\int_0^t \Big[\int_0^1 &\nabla_x b(s, \omega, \Xi(s)) d\xi - \nabla_x b(s, \omega, X_\theta(s)) \Big] \cdot \Big[ \tfrac{X_{\theta+h}(s) - X_\theta(s)}{\|h\|_{L^p}}\Big] ds
\\
\label{eq:FrechetconvProp2}
+ \int_0^t \Big[\int_0^1 &\nabla_x \sigma(s, \omega, \Xi(s)) d\xi - \nabla_x \sigma(s, \omega, X_\theta(s)) \Big] \cdot \Big[\tfrac{X_{\theta+h}(s) - X_\theta(s)}{\|h\|_{L^p}} \Big] dW(s)
\\
\nonumber
+\int_0^t \nabla_x &b(s, \omega, X_\theta(s)) \Big[ \tfrac{ X_{\theta+h}(s) - X_\theta(s) - F(s)[h](s) }{\|h\|_{L^p}} \Big] ds
\\
\nonumber
+ \int_0^t \nabla_x &\sigma(s, \omega, X_\theta(s))\Big[ \tfrac{X_{\theta+h}(s) - X_\theta(s) - F(s)[h]}{\|h\|_{L^p}} \Big] dW(s).
\end{align} 
Arguing the same way as in Theorem \ref{theo:GateauxDifferentiabilitySDE}, we show that Equation \eqref{eq:FrechetconvProp1} and \eqref{eq:FrechetconvProp2} converge to zero in probability as $\|h\|_{L^p} \to 0$. Then we apply Proposition \ref{Proposition:GronwallConProb} to conclude that
\begin{align*}
\frac{\| X_{\theta+h} - X_\theta - F[h]\|_\infty}{\|h\|_{L^p}} \xrightarrow{\bP} 0.
\end{align*}
Finally, from Theorem \ref{theorem:ExistenceUniquenessMao} and Theorem \ref{thm:FrechetExistence} we have that
\begin{align*}
\frac{\bE\Big[ \| X_{\theta+h} - X_\theta \|_\infty^p\Big]}{\|h\|^p_{L^p}} = O (1), 
\quad
\frac{\bE\Big[ \|F[h]\|_{\infty}^p\Big] }{\|h\|_{L^p}^p} = O(1)
\qquad \textrm{as}\quad \|h\|_{L^p}\to 0.
\end{align*}
Therefore, the random variable 
$
\Big\| \frac{ X_{\theta+h}(t) - X_\theta(t) - F[h](t)}{\|h\|_{L^p}} \Big\|_\infty^q 
$
is uniformly integrable and we conclude
$$
\Big\| \frac{X_{\theta+h} - X_\theta - F[h]}{\|h\|_{L^p}} \Big\|_{\cS^q} \to 0. 
$$
\end{proof}

\subsubsection*{Proof of the Frech\'et differentiability theorem}
\begin{proof}[Proof of Theorem \ref{theorem:Frechet}]
In Proposition \ref{proposition:FrechetLinear} we proved that $F$ is a bounded linear operator and in Theorem \ref{Theorem:FrechetDerivative} we proved that it satisfies Definition \ref{definition:Frechet}. 
\end{proof}

\subsection{Classical differentiability of SDEs}

For this section, we will be studying the specific case where $\theta = x$ (a constant point in $\bR^d$) and our perturbations are all in the constant function directions.  Fix $(t, \omega)\in [0,T]\times \Omega$ and consider the map $x\in\bR^d \mapsto X_x(t, \omega)$.  We will be proving that, with probability 1 and for Lebesgue almost all $t\in[0,T]$, it is a diffeomorphism from $\bR^d$ to $\bR^d$.  For this section, $h\in \bR^d$ will represent some deterministic vector in Euclidean space. We will be calculating the partial derivatives in direction $h$.

\subsubsection*{The Jacobian Matrix $J$}

\begin{definition}
Let $p\geq2$. Let $X_x$ be solution to the SDE \eqref{eq:SDE} under Assumption \ref{Assumption:Frechet} and with initial condition $X_x(0)=x\in \bR^d$. Let $I_d$ be the $d$-dimensional identity matrix. For $q\geq1$ and let $J\in \cS^q([0,T]; \bR^{d\times d})$ be the solution of the matrix valued SDE, $t\in[0,T]$
\begin{equation}
\label{eq:SDEJacobian}
J(t) = I_d + \int_0^t \nabla_x b\Big(s, \omega, X_x(s)(\omega)\Big) J(s)ds + \int_0^t \nabla_x \sigma\Big(s, \omega, X_x(s)(\omega)\Big) J(s) dW(s).
\end{equation}
\end{definition}

Notice that Equation \eqref{eq:SDEJacobian} is the same SDE as \eqref{eq:LinearSDEFundamental}. This means the Jacobian has an explicit solution which will be useful in Section \ref{section:Applications} below. 

\begin{theorem}
Let $p\geq2$. Let $X_x$ be solution to the SDE \eqref{eq:SDE} under Assumption \ref{Assumption:Frechet} and with initial condition $x\in \bR^d$. Then the SDE \eqref{eq:SDEJacobian} has a unique solution in $\cS^p$ and for any choice of $t\in [0T]$ the map $x\mapsto X_x(t)(\omega)$ is differentiable $\bP$-almost surely. The derivative is almost surely equal to the solution of the Jacobian Equation, SDE \eqref{eq:SDEJacobian}. 
\end{theorem}
 
\subsubsection*{Differentiability of $X_{x}$}

In the previous section we proved almost sure continuity of ${\|X_{x+ \varepsilon h} - X_x\|_{\infty}}/{\varepsilon}$, we need to show that the limit as $\varepsilon \to 0$ is equal to the solution of the Jacobian SDE. 

\begin{assumption}
\label{Assumption:Classical}
Let $b: [0,T]\times \Omega\times \bR^d \to \bR^d$ and $\sigma: [0,T] \times \Omega\times\bR^d \to \bR^{d\times m}$ satisfy Assumption \ref{Assumption:1} for some $p\geq2$. Further, suppose that $\nabla_x b: [0,T]\times \Omega\times \bR^d \to \bR^{d\times d}$ and  $\nabla_x \sigma: [0,T]\times \Omega\times \bR^d \to \bR^{d\times m\times d}$ are progressively measurable and that
\begin{enumerate}[(i)]
\item For almost all $(t, \omega)\in [0,T]\times \Omega$ we have the functions $\sigma(t, \omega, \cdot)$ and $b(t, \omega, \cdot)$ have partial derivatives in all directions. 
\item For $x\in \bR^d$, we have that the maps $\bR^d \to L^0(\Omega)$ 
\begin{align*}
x \mapsto \int_0^T \Big| \nabla_x \sigma(t, \omega, x) \Big|^2 dt
\qquad\textrm{and}\qquad 
x \mapsto \int_0^T \Big| \nabla_x b(t, \omega, x) \Big|^2 dt,
\end{align*}
are continuous (where convergence in $L^0$ means convergence in probability).
\item For almost all $(t, \omega)\in[0,T]\times \Omega$ we have
$$
|\nabla_x \sigma(t, \omega, x) - \nabla_x\sigma(t, \omega, y) | \leq L |x-y|. 
$$

\item For $x, y\in\bR^d$ such that $|x|, |y|<N$ and for almost all $(s, \omega)\in [0,T]\times \Omega$, $\exists L_N>0$ such that
\begin{align*}
| \nabla_x b(s, \omega, x) - \nabla_x b(s, \omega, y) | \leq L_N |x-y|.
\end{align*} 
\end{enumerate}
\end{assumption}

\begin{proposition}
\label{pro:ClassicalContinuity}
Let $p\geq2$. Let $X_x$ be solution to the SDE \eqref{eq:SDE} under Assumption \ref{Assumption:Classical} and with initial condition $x\in \bR^d$. Then we have that the map
\begin{align*}
\varepsilon \mapsto \Big\| \frac{X_{x+ \varepsilon h}(\omega) - X_x(\omega)}{\varepsilon} \|_\infty
\quad \textrm{is $\bP$-almost surely continuous around 0}.
\end{align*}
\end{proposition}

\begin{proof}
By the Stochastic Stability from Theorem \ref{theorem:ExistenceUniquenessMao} we have
$\bE\big[ \| X_{x} - X_{y} \|_\infty^p\big] \lesssim |x -y|^p$, hence by the Kolmogorov Continuity Criterion we have that the map $\varepsilon \mapsto X_{x+\varepsilon h}$ is almost surely continuous. In fact, one can show $\alpha$-H\"older continuity for $\alpha<1$ but not for when $\alpha = 1$ (which would imply Lipschitz Continuity). Therefore, we additionally need to prove almost sure continuity of the map $\varepsilon \mapsto ({X_{x+\varepsilon h} - X_x})/{\varepsilon}$. 

Denote for any $t\in[0,T]$ the auxiliary process $K_\varepsilon(t) = ({X_{x+\varepsilon h}(t) - X_x(t)})/{\varepsilon}$. This process satisfies the Linear SDE
\begin{align*}
K_\varepsilon(t) = h &+ \int_0^t \Big[ \int_0^1 \nabla_x b\big(s, \omega, X_{x}(s)+ \xi[X_{x+\varepsilon h}(s) - X_{x}(s)]\big)d\xi \Big] K_\varepsilon(s) ds 
\\
&+ \int_0^t \Big[ \int_0^1 \nabla_x \sigma\big(s, \omega, X_{x}(s)+ \xi[X_{x+\varepsilon h}(s) - X_{x}(s)]\big) d\xi \Big] K_\varepsilon(s) dW(s),
\end{align*}
and, introducing the auxiliary process $\Xi_\varepsilon(\cdot) := X_{x}(\cdot)+ \xi[X_{x+\varepsilon h}(\cdot) - X_{x}(\cdot)]\big)$, we can write the explicit solution of $K_\varepsilon$ (as it is the solution a geometric Brownian motion type SDE) 
\begin{align}
\label{eq:ContinuityJacobian1}
K_\varepsilon(t) = h \cdot \exp\Bigg( \int_0^t \Big[ \int_0^1 \nabla_x b\Big(s, \omega, \Xi_\varepsilon(s)\Big)d\xi\Big] ds\Bigg) 
\cdot \cE\Bigg( \int_0^1 \nabla_x \sigma\Big(\cdot, \omega, \Xi_\varepsilon(\cdot) \Big) d\xi\Bigg)(t),
\end{align}
where $\cE$ is the Dol\'ean-Dade operator introduced in \eqref{eq:DDexponential}, which for shorthand notation we denote $K_\varepsilon'(t) = \cE\big( \int_0^1 \nabla_x \sigma\big(\cdot, \omega, \Xi_\varepsilon(\cdot) \big) d\xi\big)(t)$. 

We now analyze the behaviour of differences of increments of $K_\varepsilon'$ in $\varepsilon$ parameter. Take $\delta>0$, using the properties of the Dol\'ean-Dade exponential, we have 
\begin{align*}
K_\varepsilon'(t) - K_\delta'(t) =
& 
\int_0^t \Big[ \int_0^1 \nabla_x \sigma\Big(s, \omega, \Xi_\varepsilon(s) \Big) 
                      - \nabla_x \sigma\Big(s, \omega, \Xi_\delta(s) \Big) d\xi \Big] K_\varepsilon'(s) dW(s)
\\
&+ \int_0^t \Big[ \int_0^1 \nabla_x \sigma\Big(s, \omega, \Xi_\delta(s) \Big)d\xi\Big] 
                                      \cdot \Big[K_\varepsilon'(s) - K_\delta'(s)\Big] dW(s). 
\end{align*}
Applying It\^o's formula for $f(x)=|x|^p$ and denoting $L(\cdot) = K_\varepsilon'(\cdot) - K_\delta'(\cdot)$ we get
\begin{align}
\label{eq:ContinuityJacobian1.1}
|L(t)|^p &= p\int_0^t | L(s)|^{p-2} L(s)^T \cdot \Big[ \int_0^1 \nabla_x \sigma\Big(s, \omega, \Xi_\varepsilon(s) \Big) - \nabla_x \sigma\Big(s, \omega, \Xi_\delta(s) \Big) d\xi \Big] K_\varepsilon'(s) dW(s)
\\
\label{eq:ContinuityJacobian1.2}
&+p\int_0^t | L(s)|^{p-2} L(s)^T \cdot \Big[ \int_0^1 \nabla_x \sigma\Big(s, \omega, \Xi_\delta(s) \Big)d\xi\Big] \cdot L(s) dW(s)
\\
\label{eq:ContinuityJacobian1.3}
&+\frac{p}{2} \int_0^t |L(s)|^{p-2} \Big| \big[ \int_0^1 \nabla_x \sigma\Big(s, \omega, \Xi_\varepsilon(s) \Big) - \nabla_x \sigma\Big(s, \omega, \Xi_\delta(s) \Big) d\xi \big] K_\varepsilon'(s)\Big|^2 ds
\\
\label{eq:ContinuityJacobian1.4}
&+\frac{p}{2} \int_0^t |L(s)|^{p-2} \Big| \big[\int_0^1 \nabla_x \sigma\Big(s, \omega, \Xi_\delta(s) \Big)d\xi\big] \cdot L(s) \Big|^2 ds
\\
\label{eq:ContinuityJacobian1.5}
&+\frac{p(p-2)}{2} \int_0^t |L(s)|^{p-4} \Big| L(s)^T \big[ \int_0^1 \nabla_x \sigma\Big(s, \omega, \Xi_\varepsilon(s) \Big) - \nabla_x \sigma\Big(s, \omega, \Xi_\delta(s) \Big) d\xi \big] K_\varepsilon'(s) \Big|^2 ds
\\
\label{eq:ContinuityJacobian1.6}
&+\frac{p(p-2)}{2} \int_0^t |L(s)|^{p-4} \Big| L(s)^T \big[\int_0^1 \nabla_x \sigma\Big(s, \omega, \Xi_\delta(s) \Big)d\xi\big] \cdot L(s) \Big|^2 ds. 
\end{align}

Next, take a supremum over time then expectations. Using the methods that have already been explored in detail for the proof of Theorem \ref{theo:ExistenceUniqueness for Candidate DX}, we know that the terms from lines \eqref{eq:ContinuityJacobian1.2}, \eqref{eq:ContinuityJacobian1.4} and \eqref{eq:ContinuityJacobian1.6} will all yield terms of the form $\lesssim \int_0^T \bE\big[ \|L\|_{\infty, t}^p\big] dt $ which will be accounted for with the Gr\"onwall inequality. 

Firstly, following the same methods for Theorem \ref{theo:ExistenceUniqueness for Candidate DX} and using the additional Assumption \ref{Assumption:Classical}(iii), for
\begin{align*}
\textrm{for }\eqref{eq:ContinuityJacobian1.1} \Rightarrow
& p \bE\Big[ \sup_{t\in[0,T]} \int_0^t |L(s)|^{p-2} \cdot \Big[\int_0^1 \nabla_x \sigma\Big(s, \omega, \Xi_\varepsilon(s) \Big) - \nabla_x \sigma\Big(s, \omega, \Xi_\delta(s) \Big) d\xi \Big] K_\varepsilon'(s) dW(s) \Big],
\\
&\leq pC_1 \bE\Big[ \Big( \int_0^T |L(s)|^{2p-4} \cdot \Big| \big[\int_0^1 \nabla_x \sigma\Big(s, \omega, \Xi_\varepsilon(s) \Big) - \nabla_x \sigma\Big(s, \omega, \Xi_\delta(s) \Big) d\xi \big] K_\varepsilon'(s) \Big|^2 ds \Big)^{\tfrac{1}{2}} \Big], 
\\
&\leq \frac{\bE\Big[ \|L\|_\infty^p\Big]}{n} + C_1^p\big[n(p-1)\big]^{p-1} \bE\Big[ \Big( TL \| X_{x+\varepsilon h} - X_{x+\delta h}\|_\infty^2 \cdot \| K_\varepsilon'\|_\infty^2 \Big)^{\tfrac{p}{2}} \Big], 
\\
&\leq \frac{\bE\Big[ \|L\|_\infty^p\Big]}{n} + C_1^p\big[n(p-1)\big]^{p-1} (TL)^{\tfrac{p}{2}} \bE\Big[ \| X_{x+\varepsilon h} - X_{x+\delta h}\|_\infty^{2p}\Big]^{\tfrac{1}{2}} \cdot \bE\Big[ \| K_\varepsilon'\|_\infty^{2p} \Big]^{\tfrac{1}{2}}. 
\end{align*} 

Secondly,
\begin{align*}
\textrm{for }\eqref{eq:ContinuityJacobian1.3} \Rightarrow
& \frac{p}{2} \bE\Big[ \int_0^T |L(s)|^{p-2} \Big| \big[ \int_0^1 \nabla_x \sigma\Big(s, \omega, \Xi_\varepsilon(s) \Big) - \nabla_x \sigma\Big(s, \omega, \Xi_\delta(s) \Big) d\xi \big] K_\varepsilon'(s)\Big|^2 ds \Big], 
\\
&\leq \frac{\bE\Big[ \|L\|_\infty^p\Big]}{n} + \Big[ \frac{n(p-2)}{2}\Big]^{\tfrac{p-2}{2}} (LT)^{\tfrac{p}{2}} \bE\Big[ \|X_{x+\varepsilon h} - X_{x+\delta h}\|_\infty^{2p} \Big]^{\tfrac{1}{2}} \cdot \bE\Big[ \| K_\varepsilon'\|_\infty^{2p} \Big]^{\tfrac{1}{2}}. 
\end{align*} 
The terms from \eqref{eq:ContinuityJacobian1.5} are treated in exactly the same way. 

Finally, we use that $\bE\big[ \| K_\varepsilon'\|_\infty^{2p} \big]^{1/2}<\infty$ and $\bE\big[ \|X_{x+\varepsilon h} - X_{x+\delta h}\|_\infty^{2p} \big]^{1/2}\lesssim |\delta - \varepsilon|^p |h|^p$ to conclude
$$
\bE\Big[ \| K_\varepsilon' - K_\delta'\|_\infty^p \Big] \lesssim |\varepsilon - \delta|^p |h|^p. 
$$
Hence by Kolmogorov Continuity Criterion, we have the map $\varepsilon \mapsto K_\varepsilon'(t)(\omega)$ is almost surely continuous for any $t\in[0,T]$ $\bP$-almost surely. 

Now, we return to Equation \eqref{eq:ContinuityJacobian1}. Using the almost sure continuity of $\varepsilon\mapsto X_{x+\varepsilon h}(t)(\omega)$ and Assumption \ref{Assumption:Classical} (iv), we have that
$$
\varepsilon \mapsto \exp\Bigg( \int_0^t \Big[ \int_0^1 \nabla_x b\Big(s, \omega, X_{x}(\cdot)+ \xi[X_{x+\varepsilon h}(\cdot) - X_{x}(\cdot)]\Big)d\xi\Big] ds\Bigg), 
$$
is almost surely continuous. Hence $\varepsilon \mapsto K_\varepsilon(t)(\omega)$ is also almost surely continuous.  
\end{proof}

\begin{theorem}
\label{thm:ClassicalDifferentiability}
Let $p\geq2$. Let $X_x$ be solution to the SDE \eqref{eq:SDE} under Assumption \ref{Assumption:Classical} and with initial condition $x\in \bR^d$. Then we have that $\forall t \in [0,T]$
\begin{align*}
\frac{X_{x+ \varepsilon h}(t)(\omega) - X_x(t)(\omega)}{\varepsilon} \to h \cdot J(t)(\omega) 
\quad \textrm{$\bP$-almost surely as $\varepsilon \to 0$}.
\end{align*}
\end{theorem}

\begin{proof}
Let $t\in[0,T]$. First, we show convergence in probability of $({X_{x+ \varepsilon h}(t) - X_x(t)})/{\varepsilon}$ to $h\cdot J(t) $ using Proposition \ref{Proposition:GronwallConProb}. Convergence in probability will imply the existence of a subsequence which converges almost sure. Finally, using Proposition \ref{pro:ClassicalContinuity} we know the limit will be almost surely unique. 

Writing out the SDE for the increments' process, we have
\begin{align}
\nonumber
& \frac{X_{x+\varepsilon h}(t)- X_{x}(t)}{\varepsilon}  - hJ(t)
\\ 
\label{eq:JacobConProb:1}
&\qquad = \int_0^t \Big[ \int_0^1 \nabla_x b(s, \omega, \Xi_\varepsilon(s))d\xi - \nabla_x b(s, \omega, X_x(s))\Big] \Big[ \frac{X_{x+\varepsilon h}(s) - X_x(s)}{\varepsilon} \Big] ds
\\
\label{eq:JacobConProb:2}
&\qquad\quad +\int_0^t \Big[ \int_0^1 \nabla_x \sigma(s, \omega, \Xi_\varepsilon(s)) d\xi - \nabla_x\sigma(s, \omega, X_x(s)) \Big] \Big[ \frac{X_{x+\varepsilon h}(s) - X_x(s)}{\varepsilon} \Big] dW(s)
\\
\nonumber
&\qquad \quad+ \int_0^t \nabla_xb(s, \omega, X_x(s)) \Big[ \frac{X_{x+\varepsilon h}(s) - X_{x}(s)}{\varepsilon} - hJ(s)\Big] ds
\\
\nonumber
&\qquad \quad+ \int_0^t \nabla_x\sigma(s, \omega, X_x(s)) \Big[ \frac{X_{x+\varepsilon h}(s) - X_{x}(s)}{\varepsilon} - hJ(s)\Big] dW(s),
\end{align}
where $\Xi_\varepsilon(\cdot) = X_{x}(\cdot)+ \xi[X_{x+\varepsilon h}(\cdot) - X_{x}(\cdot)]$. 
As with Theorem \ref{theo:GateauxDifferentiabilitySDE}, we argue that the terms \eqref{eq:JacobConProb:1} and \eqref{eq:JacobConProb:2} converge in probability to $0$, then use Proposition \ref{Proposition:GronwallConProb} to conclude that
\begin{align*}
\Big\| \frac{ X_{x+\varepsilon h}(\omega)(\cdot) - X_{x}(\omega)(\cdot) }{\varepsilon} - hJ(\omega)(\cdot) \Big\|_\infty \xrightarrow{\bP} 0. 
\end{align*}
Thus there exists a sequence $\varepsilon_n$ such that $\varepsilon_n \to 0$ as $n\to \infty$ and an event $C_1\subset \Omega$ with $\bP[C_1]=0$ such that $\forall \omega \in \Omega\backslash C_1$

\begin{align*}
\lim_{n\to \infty} \Big\| \frac{ X_{x+\varepsilon_n h}(\omega) - X_{x}(\omega) }{\varepsilon_n} - hJ(\omega) \Big\|_\infty \rightarrow 0. 
\end{align*}

Finally, by Proposition \ref{pro:ClassicalContinuity} there exists an event $C_2 \subset \Omega$ with $\bP[C_2]=0$ such that $\forall \omega \in \Omega\backslash C_2$ the map
$$
\varepsilon \mapsto \Big\| \frac{ X_{x+\varepsilon h}(\omega) - X_{x}(\omega) }{\varepsilon} - hJ(\omega) \Big\|_\infty,
$$
is continuous for $\varepsilon$ at 0. Then for $\forall \omega \in \Omega \backslash (C_1 \cup C_2)$
$$
\lim_{\varepsilon \to 0} \Big\| \frac{ X_{x+\varepsilon h}(\omega) - X_{x}(\omega) }{\varepsilon} - hJ(\omega) \Big\|_\infty \rightarrow 0,
$$
and $\bP[ C_1 \cup C_2] = 0$. 

\end{proof}

\subsubsection*{Invertibility of the Jacobian Matrix}
Next, we wish to show that the Jacobian Matrix $J(t)$ is $\bP$-almost surely invertible for any choice of $t\in[0,T]$. Notice that due to the initial condition, we have that this is true for $t=0$ since $J(0) = I_d$. 

To prove the Jacobian is invertible, we consider a matrix valued stochastic process and observe that for any choice of $t\in[0,T]$, this process will take value equal to the left inverse of $J$. This proof follows that of Nualart, \cite{nualart2006malliavin}*{Chapter 2.3; Equation 2.8}.

We introduce the SDE 
\begin{align}
\label{eq:InverseJacobian}
\nonumber
K(t) = I_d &- \int_0^t K(s) \Big[ \nabla_x b(s, \omega, X(s)) - \Big\langle \nabla_x \sigma, \nabla_x \sigma\Big\rangle_{\bR^m}\big(s, \omega, X(s)\big)\Big]ds
\\
& \qquad  - \int_0^t K(s) \nabla_x \sigma(s, \omega, X(s)) dW(s).
\end{align}

\begin{proposition}
Let $p\geq2$. Let $X_x$ be solution to the SDE \eqref{eq:SDE} under Assumption \ref{Assumption:Frechet} and with initial condition $x\in \bR^d$. Then we have the following identity $K(t)J(t) = I_d$ for all $t\in[0,T]$ $\bP$-a.s.
\end{proposition}

\begin{proof}
We deal here with matrix valued processes which cannot necessarily be assumed commutative, this makes the analysis slightly more involved. It\^o's formula for matrices gives that $(KJ)(0) = I_d$ and 
\begin{align*}
d(K J)(t) =& K(t) dJ(t) + dK(t)J(t) + d[K, J](t),
\\
=& K(t) \nabla_x b(t, \omega, X(t)) J(t) dt + K(t) \sigma(t, \omega, X(t)) J(t) dW(t)
\\
&-K(t) \nabla_x b(t, \omega, X(t)) J(t) dt -  K(t) \sigma(t, \omega, X(t)) J(t) dW(t)
\\
&+ K(t)\Big\langle \nabla_x \sigma, \nabla_x \sigma\Big\rangle_{\bR^m}\big(s, \omega, X(s)\big) J(t) dt
\\
&- K(t)\Big\langle \nabla_x \sigma, \nabla_x \sigma\Big\rangle_{\bR^m}\big(s, \omega, X(s)\big) J(t) dt = 0dt + 0dW(t).
\end{align*}
\end{proof}
SDE \eqref{eq:InverseJacobian} does not necessarily satisfy Assumption \ref{Assumption:linear} as the term $-z^T\nabla_x b(t, \omega, X(t))z$ is not bounded from above by a constant almost surely for any choice of vector $|z|=1$. However, an explicit solution to the SDE can be written out pathwise, even if it does not have finite moments. This construction will have the property that it is the left inverse of $J$.

\begin{proposition}
The determinant of the Matrix $J(t)$, denoted $D(t)$, is called the Stochastic Wronskian and satisfies the SDE
\begin{align}
\label{eq:WronskianSDE}
dD(t) =& \mbox{Tr}\Big(\nabla_x b(t, \omega, X(t)) \Big) D(t) dt + \mbox{Tr}\Big( \nabla_x \sigma(t, \omega, X(t))\Big) D(t) dW(t)
\\
\nonumber
&+ \Big[ \Big\langle \mbox{Tr}\big( \nabla_x\sigma(t, \omega, X(t)) \big), \mbox{Tr}\big(\nabla_x\sigma(t, \omega, X(t)) \big) \Big\rangle_{\bR^m} 
\\
&
\qquad \qquad \nonumber
- \mbox{Tr}\Big( \big\langle \nabla_x\sigma(t, \omega, X(t)), \nabla_x\sigma(t, \omega, X(t))\big\rangle_{\bR^m}\Big) \Big] D(t) dt, 
\end{align}
with $D(0) = 1$. $D(t)$ has explicit form
\begin{align}
\label{eq:WronskianExplicit}
D(t) = \exp&\Bigg( \int_0^t \mbox{Tr}\Big( \nabla_x b(s, \omega, X(s))\Big) - \frac{1}{2} \mbox{Tr}\Big(\Big\langle \nabla_x \sigma(s, \omega, X(s)) , \nabla_x \sigma(s, \omega, X(s))\Big\rangle_{\bR^m} \Big)ds
\\
\nonumber
&+ \int_0^t \mbox{Tr} \Big( \nabla_x \sigma(s, \omega, X(s)) \Big)dW(s) \Bigg).
\end{align}
\end{proposition}

\begin{proof}
The proof can be found in \cite{mao2008stochastic}*{Theorem 3.2.2}. The proof involves applying It\^o's formula to the determinant of $J(t)$ and establishing that it satisfies Equation \eqref{eq:WronskianSDE}. Then one applies It\^o's formula to Equation \eqref{eq:WronskianExplicit} and verifies that this likewise satisfies \eqref{eq:WronskianSDE}. Finally, by Theorem \ref{theorem:ExistenceUniquenessMaoLinearSDE}, the solution is unique. 
\end{proof}

The matrix $\nabla_xb$ being lower semidefinite means that Tr$(\nabla_xb)$ is bounded  from above, but not necessarily from below. We can conclude the $D(\cdot)$ is almost surely positive and therefore the process $K$ is $\bP$-almost surely the inverse (left or right) of $J$ provided Tr$(\nabla_xb)\neq -\infty$ with positive probability.


\section{Applications}
\label{section:Applications}

In this section, we recover and discuss some standard applications of Malliavin Differentiation and evaluate some of the problems that occur under our framework. 

\subsection{Representation formulae}
Firstly, we present a way of writing the Malliavin Derivative of $X_\theta$ in terms of the Jacobian. 
\begin{proposition}[Representation formulae]
Let $X_x$ be solution to the SDE \eqref{eq:SDE} under Assumption \ref{Assumption:2} and with initial condition $X_x(0)=x\in \bR^d$. Let $J$ satisfy the SDE \eqref{eq:SDEJacobian}. Consider the SDE for the process $J(t)J(s)^{-1}$ for $t>s$. 
\begin{align}
\nonumber
J_s(t) =& J(t) J(s)^{-1} 
\\
\nonumber
=& J(s)J(s)^{-1}  + \int_s^t \nabla_x b(r, \omega, X(r)(\omega)) J(r)J(s)^{-1} dr 
\\
&
\nonumber
\qquad \qquad \qquad \qquad + \int_s^t \nabla_x \sigma(r, \omega, X(r)(\omega)) J(r) J(s)^{-1} dW(r)
\\
\label{eq:JacobianFundMatrix}
=& I_d + \int_s^t \nabla_x b(r, \omega, X(r)(\omega)) J_s(r)dr + \int_s^t \nabla_x \sigma(r, \omega, X(r)(\omega)) J_s(r) dW(r).
\end{align}
Equation \eqref{eq:JacobianFundMatrix} is the Fundamental Matrix of the Linear Stochastic Differential Equation \eqref{eq:SDEMallDeriv}. As such, under Assumption \ref{Assumption:2} the Malliavin Derivative of $X$ can be expressed for $t>s$ as
\begin{align*}
D_sX(t) = J_s(t) A(s, t), 
\end{align*} 
where $A(s, t)$ is defined for $t>s$ as 
\begin{align*}
A(s, t)=  \sigma(s, \omega, X(s)(\omega)) &+ \int_s^t J_s(r)^{-1}\Big( U(s, r, \omega) - 
\Big\langle \nabla_x \sigma(r, \omega, X(r)(\omega)),  V(s, r, \omega)\Big\rangle_{\bR^m} \Big) dr
\\
&+ \int_s^t J_s^{-1}(r) V(s, r, \omega) dW(r).
\end{align*} 
\end{proposition}
\begin{proof}
The proof of this representation formula follows the same ideas as Theorem \ref{theo:ExistenceUniqueness for Candidate DX}. Equation \eqref{eq:SDEMallDeriv} is an infinite dimensional SDE, so we project from the infinite dimensional space into a finite dimensional space. We follow the method of \cite{mao2008stochastic}*{Theorem 3.3.1} to solve the solution explicitly in the projection space then use the Dominated Convergence Theorem to ensure the passage to the limit.

\end{proof}

\subsubsection*{Absolute Continuity}

In \cite{nualart2006malliavin}*{Theorem 2.3.1}, it is proved that the solution of a Stochastic Differential Equation with Lipschitz,  deterministic coefficients and elliptic diffusion term has a law which is absolutely continuous with respect to Lebesgue measure on $\bR^d$. This proof can be easily extended to the case where the drift term has monotone growth. 

\begin{theorem}
Let $X_x$ be solution to the SDE \eqref{eq:SDE} under Assumption \ref{Assumption:2} and with initial condition $X_x(0)=x\in \bR^d$. Suppose additionally that $\forall z\in\bR^d$ that 
\begin{align*}
z^TA(s, t) A(s, t)^Tz > \lambda(s, t) |z|^2\geq0, \quad \int_0^t \lambda(s, t) ds >0 \quad\bP\mbox{-almost surely.} 
\end{align*}
Then the law of $X_x(t)$ is absolutely continuous with respect to the Lebesgue measure on $\bR^d$. 
\end{theorem}

\begin{proof}
For this proof, recall \cite{nualart2006malliavin}*{Corollary 2.1.2} and following that our strategy is to show that the Malliavin matrix is $\bP$-almost surely non zero. 

The Malliavin Matrix, $Q(t)$ is defined to be
\begin{align*}
Q(t) = \int_0^t D_sX(t) D_sX(t)^T ds = J(t) \Bigg( \int_0^t K(s) A(s, t) A(s, t)^T K(s)^T ds \Bigg) J(t)^T.
\end{align*}
Therefore, for $z\in \bR^d$ we have $z^TQ(t) z \geq \int_0^t \lambda(s, t) |K(s)|^2 ds\cdot |J(t)|^2 \cdot |z|^2$ which is greater than zero because $|J|, |K|>0$. 
\end{proof}

\begin{remark}
Observe that the Ellipticity condition for $\sigma$ is no longer enough to ensure that the law is absolutely continuous. When $b$ and $\sigma$ are deterministic, $U$ and $V$ are uniformly 0 and Ellipticity is enough. 
\end{remark}

\subsection{Bismut-Elworthy-Li formula}
In \cite{Elworthy1992}, the author uses Malliavin Differentiability of an SDE $X_x$ to prove differentiability for functions of the form
$u(x) = \bE[ \phi(X_{x}(t)) ]$
where $\phi$ is assumed to be a continuous function and $t\in[0,T]$. This was later extended in \cite{Fournie1999} and \cite{Fournie2001} to cover functions $\phi$ which are integrable and even measurable (provided $u$ remains finite). 

Define for $t\in(0,T]$ the set $
\Gamma_t = \big\{ a\in L^2([0,T]); \int_0^t a(s) ds = 1 \big\}$. 

\begin{theorem}[Bismut-Elworthy-Li formula]
Let $\Phi:\bR^d \to \bR^d$ be a bounded, measurable function. Let $X_x$ be solution to the SDE \eqref{eq:SDE} under Assumption \ref{Assumption:2} and with initial condition $X_x(0)=x\in \bR^d$. Let $t\in(0,T]$. Suppose additionally that ($\delta(\cdot)$ stands for the usual Skorokhod integral, see \cite{nualart2006malliavin})
\begin{enumerate}
\item $\forall s\in [0,t]$ the matrix $A(s, t)$ has a right inverse, 
\item $\exists a \in \Gamma_t$ such that $a(\cdot)A(\cdot, t)^{-1} J(\cdot) \in \mbox{dom}(\delta)$. 
\end{enumerate}
Then
\begin{equation*}
\label{eq:BisElworthLi}
\nabla_x \bE\Big[ \Phi(X_x(t)\Big] = \bE\Big[ \Phi(X_x(t))  \delta\Big( a(s) A(s, t)^{-1} J(s) \Big) \Big].
\end{equation*}
\end{theorem}
\begin{proof}
We give only a sketch of the proof. For a more detailed proof, see \cite{Fournie1999} and \cite{Fournie2001}. First suppose that $\Phi$ is continuously differentiable with bounded derivatives, then 
\begin{align*}
\nabla_x \bE\Big[ \Phi(X_x(t))\Big] =& \bE\Big[ \nabla_x \Phi(X_x(t))\Big] = \bE\Big[ \nabla \Phi(X_x(t)) J(t) \Big]
=
\bE\Big[ \nabla \Phi(X_x(t)) D_s X_x(t) A(s, t)^{-1} J(s) \Big].
\end{align*}
Multiplying both sides by $a\in \Gamma_s$, integrating over $[0,t]$ (using $\int_0^t a(s) ds = 1$ on the LHS) and Fubini gives
\begin{align*}
\nabla_x \bE\Big[ \Phi(X_x(t))\Big] =& \bE\Big[ \int_0^t a(s) \nabla \Phi(X_x(t)) D_s X_x(t) A(s, t)^{-1} J(s) ds \Big]
\\
=& \bE\Big[ \int_0^t  D_s \Big(\Phi(X_x(t))\Big) a(s)A(s, t)^{-1} J(s) ds \Big]
=\bE\Big[ \Phi\big(X_x(t)\big) \delta\Big( a(s)A(s, t)^{-1} J(s)\Big) \Big],
\end{align*}
where in the last line we used integration-by-parts formula. 

Secondly, let $\Phi$ be bounded and measurable. Then using that $C_b^1$ is dense  in the set of bounded measurable functions, we approximate $\Phi$ by a sequence of functions $\Phi_n\in C_b^1$. Finally, using a domination argument it is shown that one can swap the limits and integrals and one reaches the conclusion. 
\end{proof}

\appendix

\section{Proofs}
\subsection{The existence and uniqueness theorem plus moment calculations}
\label{Appendix:MomentCalculations}

\begin{proof}[Proof of Theorem \ref{theorem:ExistenceUniquenessMao}]
As $p\geq 2$, we also have that
\begin{align*}
\bE\Big[ \int_0^T \Big| \sigma(s, \omega, X_\theta(s)) \Big|^2 ds\Big] \leq 2 \bE\Big[ \int_0^T \Big| \sigma(s, \omega, 0)\Big|^2 ds\Big] + 2L^2 T \bE\Big[ \| X\|_{\infty}^2 \Big] < \infty. 
\end{align*}
This means we can use \cite{oksendal2003stochastic}*{Theorem 3.2.5} to get 
$\bP$-almost sure continuity of the stochastic integral. The drift term is a Lebesgue integral so likewise is continuous in time. Hence $\bP$-almost sure continuity of $t\mapsto X_\theta(t)$ is immediate. 

Finally, let $t\in[0,T]$ and $\xi, \theta \in L^p(\cF_0; \bR^d; \bP)$. We have
\begin{align*}
X_{\xi} (t) - X_{\theta} (t) = \xi-\theta &+ \int_0^t \Big[ b(s, \omega, X_{\xi}(s)) - b(s, \omega, X_{\theta}(s))\Big] ds
\\
&+\int_0^t \Big[ \sigma(s, \omega, X_{\xi}(s)) - \sigma(s, \omega, X_{\theta}(s))\Big] dW(s). 
\end{align*}
We write $Q(s) = X_{\xi}(s) - X_\theta(s)$ and by applying It\^o's formula with $f(x) = |x|^p$ we get
\begin{align*}
|Q(t)|^p = |\xi-\theta|^p &+ p\int_0^t |Q(s)|^{p-2} \Big\langle Q(s), b(s, \omega, X_{\xi}(s)) - b(s, \omega, X_{\theta}(s)) \Big\rangle ds
\\
&+p \int_0^t |Q(s)|^{p-2} \Big\langle Q(s), \Big[ \sigma(s, \omega, X_{\xi}(s)) - \sigma(s, \omega, X_{\theta}(s)) \Big] dW(s) \Big\rangle 
\\
&+\frac{p}{2} \int_0^t |Q(s)|^{p-2} \cdot \Big| \sigma(s, \omega, X_{\xi}(s)) - \sigma(s, \omega, X_{\theta}(s)) \Big|^2 ds
\\
&+ \frac{p(p-2)}{2} \int_0^t |Q(s)|^{p-4} \Big| Q(s)^T \cdot \big[\sigma(s, \omega, X_{\xi}(s)) - \sigma(s, \omega, X_{\theta}(s))\big]\Big|^2 ds. 
\end{align*}
Taking a supremum over time and then taking expectations, we get
\begin{align}
\nonumber
\bE\Big[ &\| X_{\xi} - X_\theta\|_\infty^p \Big] = \bE\Big[ \| Q\|_\infty^p \Big] \leq \bE\Big[ |\xi-\theta|^p\Big] 
\\
\label{eq:FrechetQ1}
&+ p\bE\Big[ \int_0^T |Q(s)|^{p-2} \Big| \Big\langle Q(s), b(s, \omega, X_{\xi}(s)) - b(s, \omega, X_{\theta}(s)) \Big\rangle\Big| ds\Big]
\\
\label{eq:FrechetQ2}
&+p \bE \Big[ \sup_{t\in[0, T]} \int_0^t |Q(s)|^{p-2} \Big\langle Q(s), \Big[ \sigma(s, \omega, X_{\xi}(s)) - \sigma(s, \omega, X_{\theta}(s)) \Big] dW(s) \Big\rangle \Big]
\\
\label{eq:FrechetQ3}
&+ \frac{p}{2} \bE\Big[ \int_0^T |Q(s)|^{p-2} \Big| \sigma(s, \omega, X_{\xi}(s)) - \sigma(s, \omega, X_{\theta}(s)) \Big|^2 ds \Big]
\\
\label{eq:FrechetQ4}
&+ \frac{p(p-2)}{2} \bE\Big[ \int_0^T |Q(s)|^{p-4} | Q(s)^T \Big[\sigma(s, \omega, X_{\xi}(s)) - \sigma(s, \omega, X_{\theta}(s))\Big]|^2 ds \Big]. 
\end{align}
Firstly by monotonicity of $b$ we have 
$\eqref{eq:FrechetQ1} \leq pL \int_0^T \bE\Big[ \|Q\|_{\infty, s}^p\Big] ds$. Secondly, by the Burkholder-Davis-Gundy inequality we have
\begin{align*}
\eqref{eq:FrechetQ2} 
\leq& pC_1 L \bE\Big[ \|Q\|_{\infty}^{\tfrac{p}{2}} \Big( \int_0^T \|Q\|_{\infty, s}^p ds\Big)^{\tfrac{1}{2}} \Big] \leq \frac{\bE\Big[ \|Q\|_{\infty}^p\Big]}{2} + \frac{p^2C_1^2L^2}{2} \int_0^T \bE\Big[ \|Q\|_{\infty, s}^p \Big] ds. 
\end{align*}
Finally, we have
\begin{align*}
\eqref{eq:FrechetQ3} \leq \frac{pL}{2} \int_0^T \bE\Big[ \|Q\|_{\infty, s}^p \Big] ds
\qquad 
\textrm{and}
\qquad 
\eqref{eq:FrechetQ4}\leq \frac{p(p-2)L}{2} \int_0^T \bE\Big[ \|Q\|_{\infty, s}^p \Big] ds. 
\end{align*}
Gathering all the estimates we have finally
\begin{align*}
\frac12{\bE\Big[ \| X_{\xi} - X_\theta\|_\infty^p \Big] } 
\leq \bE\Big[ |\xi-\theta|^p\Big] + \widehat{C} \int_0^T\bE\Big[ \| X_{\xi} - X_\theta\|_{\infty, s}^p \Big] ds,
\end{align*}
where $\widehat{C}={p^2L(LC_1^2 +2)}/{2}$.  Gr\"onwall's inequality yields that 
$\bE\Big[ \| X_{\xi} - X_\theta\|_\infty^p \Big] \lesssim \bE[ |\xi-\theta|^p]$.
\end{proof}

\begin{proof}[Moment Calculations for Theorem \ref{theorem:ExistenceUniquenessMaoLinearSDE} ]
Fix $t\in [0,T]$ and using It\^o's formula with $f(x) = |x|^p$ and $X_\theta$ satisfying Equation \eqref{eq:SDEDeriv}, we get that
\begin{align*}
|X_\theta(t)|^p = |\theta|^p 
&+ p\int_0^t |X_\theta(s)|^{p-2} \langle X_\theta(s), B(s, \omega) X_\theta(s) \rangle ds
+ p\int_0^t |X_\theta(s)|^{p-2} \langle X_\theta(s), b(s, \omega) \rangle ds
\\
&+ p\int_0^t |X_\theta(s)|^{p-2} \langle X_\theta(s), \Sigma(s, \omega) X_\theta(s) dW(s)\rangle
+ p\int_0^t |X_\theta(s)|^{p-2} \langle X_\theta(s), \sigma(s, \omega) dW(s)\rangle 
\\
&+\frac{p}{2} \int_0^t |X_\theta(s)|^{p-2} \Big[ \Sigma(s, \omega) X_\theta(s) + \sigma(s, \omega)\Big]^2 ds 
\\
&+ \frac{p(p-2)}{2} \int_0^t |X_\theta(s)|^{p-4} \Big\langle X_\theta(s), \big[ \Sigma(s, \omega) X_\theta(s) + \sigma(s, \omega)\big]\Big\rangle^2 ds.
\end{align*}
Take a supremum over $t\in[0,T]$ and expectations to have
\begin{align}
\nonumber
\bE\Big[ \|X_\theta\|^p\Big] \leq& \bE\Big[ |\theta|^p \Big]
\\
\label{eq:LinearSDEExist1}
&+ p\bE\Big[ \int_0^T |X_\theta(s)|^{p-2} \Big\langle X_\theta(s), \big[B(s, \omega) X_\theta(s) + b(s, \omega)\big]\Big\rangle ds \Big]
\\
\label{eq:LinearSDEExist2}
&+ p\bE\Big[ \sup_{t\in[0,T]} \int_0^t |X_\theta(s)|^{p-2} \Big\langle X_\theta(s), \big[\Sigma(s, \omega) X_\theta(s) + \sigma(s, \omega)\big]dW(s)\Big\rangle \Big]
\\
\nonumber
&+ \frac{p(p-1)}{2} \bE\Big[ \int_0^T |X_\theta(s)|^{p-2} \Big[ \Sigma(s, \omega) X_\theta(s) + \sigma(s, \omega)\Big]^2 ds \Big].
\end{align}
Fix $n\in\bN$ to be chosen later. Throughout the next three arguments, we use Young's Inequality. Using the negative semidefinite properties of $B$, we get that
\begin{align*}
\eqref{eq:LinearSDEExist1}\leq pL \int_0^T \bE\Big[ \|X_\theta\|_{\infty, s}^p \Big] ds + \frac{\bE\Big[ \|X_\theta\|_\infty^p \Big]}{n} + n^{p-1} (p-1)^{p-1}\times \bE\Big[ \Big( \int_0^T |b(s, \omega)| ds\Big)^p \Big]. 
\end{align*}
Secondly, using the Burkholder-Davis-Gundy Inequality gives that
\begin{align*}
\eqref{eq:LinearSDEExist2}\leq& \frac{2 \bE\Big[ \|X_\theta\|_\infty^p \Big]}{n} + \frac{p^2 C_1^2 n \sqrt{2}}{4} \int_0^T \bE\Big[ \|X_\theta\|_{\infty, s}^p\Big] \| \Sigma(s, \cdot)\|_{L^{\infty}}^2 ds 
\\
&+ C_1^p n^{p-1} p^{\tfrac{p}{2}} (p-2)^{\tfrac{p-2}{2}} \frac{2^{\tfrac{p}{4}}}{2^{p-1}} \cdot \bE\Big[ \Big( \int_0^T |\sigma(s, \omega)|^2 ds \Big)^{\tfrac{p}{2}} \Big]. 
\end{align*}
Thirdly, we have
\begin{align*}
\eqref{eq:LinearSDEExist2}\leq& \frac{\bE\Big[ \|X_\theta\|_\infty^p \Big]}{n} + p(p-1)\int_0^T \bE\Big[ \|X_\theta\|_{\infty, s}^p \Big] \| \Sigma (s, \cdot)\|_{L^\infty} ds 
\\
&+ 2[n(p-2)]^{\tfrac{p-2}{2}} (p-1)^{\tfrac{p}{2}} \bE\Big[ \Big( \int_0^T |\sigma(s, \omega)|^2 ds \Big)^{\tfrac{p}{2}} \Big]. 
\end{align*}
When applying Young's Inequality for the case $p=2$, we use the convention that $0^0=1$. Adding these together, we have that there are constants $\widetilde{C_1}, \widetilde{C_2}$ and $\widetilde{C_3}$ such that
\begin{align*}
\frac15{\bE\Big[ \|X_\theta\|_{\infty}^p\Big] } \leq& \bE\Big[ |\theta|^p\Big] + \widetilde{C_1}\bE\Big[ \Big( \int_0^T |b(s, \omega)| ds\Big)^p \Big] + \widetilde{C_2} \bE\Big[ \Big( \int_0^T |\sigma(s, \omega)|^2 ds \Big)^{\tfrac{p}{2}} \Big]
\\
&+ \widetilde{C_3} \int_0^T \big[ 1+ \| \Sigma(s, \cdot)\|_{L^\infty} \big] \bE\Big[ \|X_\theta\|_{\infty, s}^p\Big] ds.
\end{align*}
Applying Gr\"onwall Inequality yields
\begin{align*}
\bE\Big[ \|X_\theta\|_{\infty}^p\Big] \leq 5& \Bigg( \bE\Big[ |\theta|^p\Big] + \widetilde{C_1}\bE\Big[ \Big( \int_0^T |b(s, \omega)| ds\Big)^p \Big] + \widetilde{C_2} \bE\Big[ \Big( \int_0^T |\sigma(s, \omega)|^2 ds \Big)^{\tfrac{p}{2}} \Big]\Bigg) 
\\
&\times \exp\Big( 5\widetilde{C_3} \int_0^T  \big[ 1+ \| \Sigma(s, \cdot)\|_{L^\infty} \big] ds\Big).
\end{align*}
\end{proof}



\bibliographystyle{amsplain}


\end{document}